\documentclass[12pt]{amsart}
\pdfoutput=1
\usepackage{graphicx,amssymb,latexsym,bm,color}
\usepackage[dvips,centering,includehead,width=15.1cm,height=22.7cm]{geometry}
\usepackage{amsmath}
\usepackage[boxed,vlined,lined]{algorithm2e}
\usepackage{graphicx}
\usepackage{amssymb}
\usepackage{rotating}
\usepackage{arydshln} 

\input{xy}
\xyoption{all}

\numberwithin{equation}{section}

\newtheorem{theorem}{Theorem}[section]
\newtheorem{lemma}[theorem]{Lemma}

\newtheorem{proposition}[theorem]{Proposition}

\newtheorem{conjecture}[theorem]{Conjecture}

\theoremstyle{definition}
\newtheorem{definition}[theorem]{Definition}
\newtheorem{example}[theorem]{Example}
\newtheorem{remark}[theorem]{Remark}

{\begin{list}{}%
         {\setlength{\leftmargin}{#1}}%
         \item[]%
}
{\end{list}}

  \newcommand{\eps}{{\varepsilon}}
  \newcommand{\e}{{\varepsilon}}

 \newcommand{\D}{{\mathcal D}}

  \newcommand{\down}{{\! \downarrow}}
  \newcommand{\defect}{{\mathrm{def}}}

 \newcommand{\wt}{{\text{wt}}}
 \newcommand{\len}{{\text{len}}}

  \newcommand{\ZZ}{{\mathbb Z}}

  \newcommand{\PP}{{\mathbb P}}

\newcommand{\oo}{\multiput(0,0)(10,0){2}{\circle{2}}}
\newcommand{\ooo}{\multiput(0,0)(10,0){3}{\circle{2}}}
\newcommand{\oooo}{\multiput(0,0)(10,0){4}{\circle{2}}}
\newcommand{\ooooo}{\multiput(0,0)(10,0){5}{\circle{2}}}
\newcommand{\oooooo}{\multiput(0,0)(10,0){6}{\circle{2}}}

\newcommand{\Eeee}{\put(1,0){\line(1,0){8}}}
\newcommand{\eEee}{\put(11,0){\line(1,0){8}}}
\newcommand{\eeEe}{\put(21,0){\line(1,0){8}}}
\newcommand{\eeeE}{\put(31,0){\line(1,0){8}}}
\newcommand{\eeeeE}{\put(41,0){\line(1,0){8}}}

\newcommand{\EEee}{\qbezier(0.8,0.6)(4,5)(10,5)\qbezier(10,5)(16,5)(19.2,0.6)}
\newcommand{\eEEe}{\qbezier(10.8,0.6)(14,5)(20,5)\qbezier(20,5)(26,5)(29.2,0.6)}

\newcommand{\BBee}{\qbezier(0.2,1)(4,10)(10,10)\qbezier(10,10)(16,10)(19.8,1)}
\newcommand{\eBBe}{\qbezier(10.2,1)(14,10)(20,10)\qbezier(20,10)(26,10)(29.8,1)}
\newcommand{\eeBB}{\qbezier(20.2,1)(24,10)(30,10)\qbezier(30,10)(36,10)(39.8,1)}
\newcommand{\eeeBB}{\qbezier(30.2,1)(34,5)(40,5)\qbezier(40,5)(46,5)(49.8,1)}

\newcommand{\eeOe}{\qbezier(20.8,0.6)(25,4)(29.2,0.6)
                  \qbezier(20.8,-0.6)(25,-4)(29.2,-0.6)}

\newcommand{\eeeeO}{\qbezier(40.8,0.6)(45,4)(49.2,0.6)
                  \qbezier(40.8,-0.6)(45,-4)(49.2,-0.6)}

  \title{Computing Node Polynomials for Plane Curves}
  \author{Florian Block}
  \date{\today}
  \address{Department of Mathematics, University of Michigan, Ann Arbor, MI 48109, USA}
  \email{blockf@umich.edu}
\thanks{The author was partially supported by the NSF grant
  DMS-055588.}


\begin{document}

\begin{abstract}
According to the G\"ottsche conjecture (now a theorem), the degree
$N^{d, \delta}$ of the Severi variety of plane curves of degree $d$
with $\delta$ nodes is given by a polynomial in $d$, provided $d$ is
large enough. These ``node polynomials'' $N_\delta(d)$ were determined by Vainsencher
and Kleiman--Piene 
for $\delta \le 6$ and $\delta \le 8$, respectively.
Building on ideas of Fomin and Mikhalkin, we develop an explicit
algorithm for computing all node polynomials, and use it to compute
$N_{\delta}(d)$ for $\delta \le 14$.
Furthermore, we improve the
threshold of
polynomiality
and verify G\"ottsche's conjecture on the optimal threshold
up to $\delta \le 14$. We also determine the first $9$ coefficients of
$N_\delta(d)$, for general $\delta$, settling and extending a 1994
conjecture of Di Francesco and Itzykson.
\end{abstract}

\maketitle

\section{Introduction and Main Results}
\label{sec:intro}


\subsection*{Node Polynomials}

Counting algebraic plane curves is a very old problem. In 1848,
J.~Steiner  determined that the number of curves of degree $d$ with $1$
node through $\frac{d(d+3)}{2} -1$ generic points in the complex
projective plane $\PP^2$ is
$3(d-1)^2$. Much effort has since been put forth towards answering the
following question:

\smallskip

\begin{center}
\emph{How many (possibly reducible) degree $d$ nodal curves with \\
  $\delta$ nodes pass through $\frac{d(d+3)}{2} - \delta$ generic
  points in $\PP^2$?}
\end{center}
\smallskip
The answer to this question is the \emph{Severi degree} $N^{d,
  \delta}$, the degree of the corresponding Severi variety.
In 1994, P.~Di Francesco and C.~Itzykson
\cite{DI} conjectured that $N^{d, \delta}$ is given by a polynomial in $d$
(assuming $\delta$ is fixed and $d$ is sufficiently large). It is not hard to see that, if such a polynomial exists, it has
to be of degree $2\delta$.

Recently, S.~Fomin and
G.~Mikhalkin \cite[Theorem 5.1]{FM} 
established the polynomiality of $N^{d, \delta}$ using tropical
geometry and floor decompositions. More precisely, they
showed that there exists, for every $\delta \ge 1$, a \emph{node polynomial}
$N_\delta(d)$ which satisfies $N^{d,\delta} = N_\delta(d)$ for all $d
\ge 2 \delta$. (The $\delta = 0$ case is trivial as $N^{d, 0} = 1$ for
all $d \ge 1$.)

For $\delta = 1, 2, 3$, the polynomiality of the Severi degrees and
the formulas for $N_\delta(d)$ were determined in the 19th
century. For $\delta = 4, 5, 6$, this was only achieved by
I.~Vainsencher \cite{Va} in 1995. In 2001, S.~Kleiman and R.~Piene
\cite{KP} settled the cases $\delta = 7, 8$. Earlier, 
L.~G\"ottsche \cite{Go}
conjectured a more detailed (still not entirely explicit) description
of these polynomials for counting nodal curves on smooth projective algebraic
surfaces.

\subsection*{Main Results}

In this paper we develop, building on ideas of S.~Fomin and
G. Mikhalkin \cite{FM}, an explicit algorithm 
(see Algorithm~\ref{alg:nodepoly}) for computing the
node polynomials $N_\delta(d)$ for arbitrary $\delta$. This
algorithm is used to calculate $N_\delta(d)$ for all $\delta
\le 14$.

\begin{theorem}
\label{thm:nodepolys}
The node polynomials $N_\delta(d)$, for $ \delta \le 14$, are as listed
in Appendix~\ref{app:newpolys}.
\end{theorem}

P.~Di Francesco and C.~Itzykson \cite{DI} conjectured the first seven
terms of the node polynomial $N_\delta(d)$, for arbitrary $\delta$. We
confirm and extend their assertion. The first two terms already appeared
in \cite{KP}.

\begin{theorem}
\label{thm:firstcoeffs}
The first nine coefficients of $N_{\delta}(d)$ are given by
\begin{displaymath}
\scriptsize
\begin{split}
N_{\delta}(d)  &= \frac{3^\delta}{\delta !} \left [ d^{2\delta}
  -2\delta d^{2\delta-1} - \frac{\delta(\delta-4)}{3}
  d^{2\delta-2}+\frac{\delta (\delta-1)(20\delta-13)}{6}d^{2\delta-3}
  + \right.\\
& - \frac{\delta (\delta-1)(69\delta^2-85\delta+92)}{54} d^{2\delta -4} -
\frac{\delta (\delta-1)(\delta-2)(702\delta^2 - 629\delta - 286)}{270}
d^{2\delta-5} + \\
& + \frac{\delta (\delta-1)(\delta-2)(6028\delta^3 - 15476 \delta^2 +
  11701\delta + 4425)}{3240} d^{2\delta - 6} + \\
& + \frac{\delta (\delta-1)(\delta-2)(\delta-3)(13628 \delta^3 -
  6089\delta^2 - 29572 \delta - 24485)}{11340} d^{2\delta -7} + \\
& \left.- \frac{\delta (\delta-1)(\delta-2)(\delta-3)(282855\delta^4-931146\delta^3 + 417490 \delta^2 +
  425202\delta + 1141616)}{204120} d^{2\delta - 8} + \cdots \right ].
\end{split}
\end{displaymath}
\end{theorem}


Let $d^*(\delta)$ denote the \emph{polynomiality threshold} for Severi
degrees, i.e., the smallest positive integer $d^* = d^*(\delta)$ such
that $N_\delta(d) = N^{d, \delta}$ for $d \ge d^*$. As mentioned above
S.~Fomin and G.~Mikhalkin showed that $d^* \le 2 \delta$. We improve
this as follows:

\begin{theorem}
\label{thm:threshold}
For $\delta \ge 1$, we have $d^*(\delta) \le \delta$.
\end{theorem}

In other words, $N^{d, \delta} = N_\delta(d)$ provided $d \ge \delta
\ge 1$. L.~G\"ottsche \cite [Conjecture 4.1]{Go} conjectured that $d^*
\le \left \lceil\frac{\delta}{2} \right \rceil + 1$ for $\delta \ge 1$. This was verified for
$\delta \le 8$ by S. Kleiman and R. Piene \cite{KP}. By direct computation we can push it further.

\begin{proposition}
\label{prop:Gthreshold}
For $3 \le \delta \le 14$, we have $d^*(\delta) = \left \lceil\frac{\delta}{2} \right \rceil + 1$.
\end{proposition}

That is, G\"ottsche's threshold is correct and sharp for $3 \le \delta
\le 14$. For $\delta = 1, 2$ it is easy to see that $d^*(1) = 1$ and
$d^*(2) = 1$.

P.~Di Francesco and C.~Itzykson \cite{DI} hypothesized that
$d^*(\delta) \le
\left \lceil
\frac{3}{2} +\sqrt{2\delta+\frac{1}{4}} \right \rceil$ 
(which is
equivalent to $\delta \le \frac{(d^*-1)(d^*-2)}{2}$). However, our
computations show that this fails for $\delta = 13$ as $d^*(13) = 8$.



The main techniques of this paper are combinatorial. By the celebrated
Correspondence Theorem of G.~Mikhalkin
\cite[Theorem 1]{Mi03} one can replace the algebraic curve count by an enumeration
of certain \emph{tropical curves}. E.~Brugall\'e and G.~Mikhalkin
\cite{BM1, BM2} introduced some purely combinatorial gadgets, called
{\it (marked) labeled floor diagrams} (see
Section~\ref{sec:floordiagrams}), which, if counted correctly, are equinumerous to these
tropical curves. Recently, S.~Fomin and G.~Mikhalkin \cite{FM}
enhanced Brugall\'e and Mikhalkin's definition and introduced a
\emph{template
decomposition} of labeled floor diagrams which is crucial in the proofs
of all results in this paper, as is the reformulation of algebraic
plane curve counts in terms of labeled floor diagrams (see Theorem~\ref{thm:floorcount}).

This paper is organized as follows: In Section~\ref{sec:floordiagrams}
we review labeled floor diagrams, their markings, and their
relationship with 
the enumeration of plane algebraic curves.
The proofs of Theorems~\ref{thm:nodepolys} and~\ref{thm:firstcoeffs}
are algorithmic in nature and involve a computer computation. We
describe both algorithms in detail in Sections~\ref{sec:nodepolys}
and~\ref{sec:coefficients}, respectively. The first algorithm computes
the node
polynomials $N_\delta(d)$ for arbitrary $\delta$, the second
determines a prescribed number of leading terms of
$N_\delta(d)$. The latter algorithm relies on the polynomiality of
solutions of certain polynomial
difference equations: This polynomiality has been verified for
pertinent values of $\delta$ (see Section~\ref{sec:coefficients}). 
Proposition~\ref{prop:Gthreshold} is proved in Section~\ref{sec:nodepolys} by
comparison of the numerical values of $N_\delta(d)$ and $N^{d,
  \delta}$ for various $d$ and $\delta$ (see
Appendices~\ref{app:newpolys} and \ref{app:smallSevdegs}).
Theorem~\ref{thm:threshold} is proved in Section~\ref{sec:thresholdvalues}.

\subsection*{Competing Approaches: Floor Diagrams vs.\ Caporaso-Harris
recursion}

An alternative approach to computing the node polynomials
$N_\delta(d)$ combines polynomial interpolation with the Caporaso-Harris
recursion~\cite{CH98}. Once a polynomiality threshold
$d_0(\delta)$ has been established (i.e., once we have proved that
$N_\delta(d) = N^{d, \delta}$ for $d \ge d_0(\delta)$), we can use the
recursion to determine a sufficient number of Severi degrees $N^{d,
  \delta}$ for $d\ge d_0(\delta)$, from  which we then interpolate.

This approach was first used by L.~G\"ottsche~\cite[Remark~4.1(1)]{Go}.
He conjectured \cite[Conjecture~4.1]{Go}
the polynomiality threshold 
$d_0(\delta)=\lceil \frac{\delta}{2} \rceil+ 1$,
and combined it with 
the ``G\"ottsche-Yau-Zaslow formula''
\cite[Conjecture~2.4]{Go} (now a theorem of Y.-J.\ Tzeng~\cite{Tz10})
to calculate the putative node polynomials $N_\delta(d)$ for $\delta
\le 28$. The G\"ottsche-Yau-Zaslow formula is a stronger version of
polynomiality that allows one to compute each next node polynomial by
calculating only two additional Severi degrees
$N^{d_0(\delta), \delta}$ and $N^{d_0(\delta)+1, \delta}$, 
which is done via the Caporaso-Harris formula. 
Since G\"ottsche's threshold $d_0(\delta) =
\lceil \frac{\delta}{2} \rceil + 1$ remains open as of this writing,
the algorithm he used to compute the node polynomials
is still awaiting a rigorous justification. 

The first polynomiality threshold 
$d_0(\delta) = 2 \delta$ was established by S.~Fomin and
G.~Mikhalkin~\cite[Theorem~5.1]{FM}.
Using this result, one can compute $N_\delta(d)$ for
$\delta \le 9$ but hardly any
further\footnote{We used an efficient C implementation of the
  Caporaso-Harris recursion by A.~Gathmann.}.
With the threshold $d_0(\delta) = \delta$ established in 
Theorem~\ref{thm:threshold}, it should be possible to compute 
$N_\delta(d)$ for $\delta\le 16$ or perhaps $\delta\le 17$.

By contrast, our Algorithm~\ref{alg:nodepoly} does not involve
interpolation nor does it require an \emph{a priori} knowledge of a
polynomiality threshold. 
Our computations verify the results of 
L.~G\"ottsche's calculations for $\delta \le 14$.
In our implementations,
Algorithm~\ref{alg:nodepoly} is roughly as efficient as the
interpolation 
method discussed above. 
(We repeat that the latter method depends on 
the threshold obtained using floor diagrams.)

\subsection*{Gromov-Witten invariants}

The \emph{Gromov-Witten invariant} $N_{d,g}$ enumerates
irreducible plane curves of degree~$d$ and genus~$g$ through
$3d+g-1$ generic points in~$\PP^2$. Algorithm~\ref{alg:nodepoly} (with
minor adjustments, cf.\,Theorem \ref{thm:floorcount}(2)) can be used to directly compute
$N_{d,g}$, without resorting to a recursion involving relative
Gromov-Witten invariants \emph{\`a la} Caporaso--Harris \cite{CH98}.

\subsection*{Follow-up work}

By extending ideas of S.~Fomin and G.~Mikhalkin \cite{FM} and of the present paper,
we can obtain polynomiality results for
\emph{relative Severi degrees}, the degrees of \emph{generalized
  Severi varieties} (see ~\cite{CH98, Va00}).  This is discussed in the
separate paper~\cite{FB10}; see Remark~\ref{rmk:relativeremark}.

A.~Gathmann, H.~Markwig and the author ~\cite{BGM10} defined Psi-floor
diagrams which enumerate
plane curves satisfying point and tangency conditions as well as 
conditions given by \emph{Psi-classes}. We prove a Caporaso-Harris type recursion
for Psi-floor diagrams, and show that 
relative descendant Gromov-Witten invariants equal their tropical counterparts.


\subsection*{Acknowledgements}
I am thankful to Sergey Fomin for suggesting this problem and fruitful guidance. I
also thank the  anonymous referee, Erwan Brugall\'e, Grigory Mikhalkin
and Gregg Musiker for valuable
comments and suggestions.  Part of this work was accomplished at the MSRI (Mathematical
  Sciences Research Institute) in Berkeley, CA, USA, during the
  semester program on tropical geometry. I thank MSRI for hospitality.

\section{Labeled Floor Diagrams}
\label{sec:floordiagrams}

Labeled floor diagrams are combinatorial gadgets which, if counted
correctly, enumerate plane curves with certain prescribed
properties. E.~Brugall\'e and G.~Mikhalkin
introduced them in
\cite{BM1} (in slightly different notation) and studied them further in \cite{BM2}. To keep this paper
self-contained and to fix notation we review them and their markings following
\cite{FM} where the framework that best suits our purposes was introduced.

\begin{definition}
A \emph{labeled floor diagram} $\D$ on a vertex set $\{1, \dots, d\}$ is a directed graph (possibly with 
multiple edges) with positive integer edge weights $w(e)$ satisfying:
\begin{enumerate}
\item The edge directions respect the order of the
  vertices, i.e., for each edge $i \to j $ of $\D$ we have $i
  < j$.
\item (Divergence Condition) For each vertex $j$ of $\D$, we have 
\[
\text{div}(j) \stackrel{\text{def}}{=} \sum_{ \tiny
     \begin{array}{c}
  \text{edges }e\\
j \stackrel{e}{\to} k
     \end{array}
} w(e) -   \sum_{ \tiny
     \begin{array}{c}
  \text{edges }e\\
i \stackrel{e}{\to} j
     \end{array}
} w(e)\le 1.
\]
\end{enumerate}
This means that at every vertex of $\D$ the total weight of the outgoing edges
is larger by at most 1 than the total weight of the incoming edges.
\end{definition}

The \emph{degree} of a labeled floor
diagram $\D$ is the number of its vertices. It is
\emph{connected} if its underlying graph is. Note that in \cite{FM}
labeled floor diagrams are required to be connected. If $\D$ is
connected its \emph{genus} 
is
the genus of the underlying graph
 (or the first Betti number of the underlying topological space). The \emph{cogenus}
of a connected labeled floor diagram $\D$
of degree $d$ and genus $g$ is given by
$\delta(\D) = \frac{(d-1)(d-2)}{2} - g$. If $\D$ is not connected, let
$d_1, d_2, \dots$ and $\delta_1, \delta_2, \dots$ be the degrees and
cogenera, respectively, of its connected components. Then the \emph{cogenus} of $\D$
is $\sum_j \delta_j + \sum_{j < j'} d_j d_{j'}$. Via the
correspondence between algebraic curves and labeled floor diagrams
(\cite[Theorem 3.9]{FM}) these notions correspond literally to
the respective analogues for algebraic curves. Connectedness
corresponds to irreducibility. Lastly, a labeled floor
diagram $\D$ has \emph{multiplicity}\footnote{If floor diagrams are
  viewed as floor contractions of tropical plane curves this
corresponds to the notion of multiplicity of tropical plane
curves.}
\[
\mu(\D) = \prod_{\text{edges }e} w(e)^2.
\]

We draw labeled floor diagrams using the convention that vertices in
increasing order are arranged left to right. Edge weights of $1$
are omitted.

\begin{example}
\label{ex:floordiagram}
An example of a labeled floor diagram of degree $d = 4$, genus $ g=1$,
 cogenus $\delta = 2$, divergences $1,1,0,-2$, and multiplicity $\mu =
 4$ is drawn below.
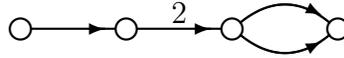
\begin{equation*}
\label{eq:floordiagram}
\begin{picture}(50,34)(30,-15)\setlength{\unitlength}{4pt}\thicklines
\oooo\Eeee\eEee\eeOe
\put(15,1.5){\makebox(0,0){$2$}} 
\put(7,0){\vector(1,0){1}} 
\put(17,0){\vector(1,0){1}} 
\put(27.5,1.75){\vector(2,-1){1}}
\put(27.5,-1.75){\vector(2,1){1}}
\end{picture}
\end{equation*}
\end{example}

To enumerate algebraic curves via labeled floor
diagrams we need the notion of markings of such diagrams.

\begin{definition}
A \emph{marking} of a labeled floor diagram $\D$ is defined by the following
three step process which we illustrate in the case of Example~\ref{ex:floordiagram}.

{\bf Step 1:} For each vertex $j$ of $\D$ create $1- div(j)$ many new vertices and
connect them to $j$ with new edges directed away from $j$.

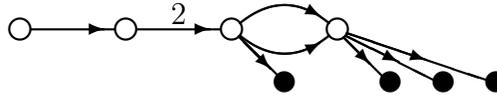
\begin{center}
\begin{picture}(50,28)(40,-19)\setlength{\unitlength}{4pt}\thicklines
\oooo\Eeee\eEee\eeOe
\put(15,1.5){\makebox(0,0){$2$}} 
\put(7,0){\vector(1,0){1}} 
\put(17,0){\vector(1,0){1}} 
\put(27.5,1.75){\vector(2,-1){1}}
\put(27.5,-1.75){\vector(2,1){1}}

\put(20.7,-0.7){\line(1,-1){4}}
\put(20.7,-0.7){\vector(1,-1){3}}
\put(25,-5){\circle*{2}}
\put(30.6,-0.8){\line(1,-1){4}}
\put(30.8,-0.6){\line(2,-1){9}}
\put(30.9,-0.4){\line(3,-1){14}}
\put(30.6,-0.8){\vector(1,-1){2.5}}
\put(30.8,-0.6){\vector(2,-1){5}}
\put(30.9,-0.4){\vector(3,-1){7.5}}
\put(35,-5){\circle*{2}}
\put(40,-5){\circle*{2}}
\put(45,-5){\circle*{2}}
\end{picture}
\end{center}

{\bf Step 2:} Subdivide each edge of the original labeled floor
diagram $\D$ into two
directed edges by introducing a new
vertex for each edge. The new edges inherit their weights and orientations. Call the
resulting graph $\tilde{\D}$.

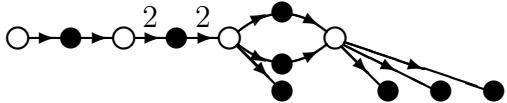
\begin{center}
\begin{picture}(50,34)(45,-20)\setlength{\unitlength}{4pt}\thicklines
\oooo\Eeee\eEee\eeOe
\put(12.5,2){\makebox(0,0){$2$}}
\put(17.5,2){\makebox(0,0){$2$}}
\put(2.5,0){\vector(1,0){1}}
\put(7.5,0){\vector(1,0){1}}
\put(12.5,0){\vector(1,0){1}}
\put(17.5,0){\vector(1,0){1}}
\put(27.5,1.75){\vector(2,-1){1}}
\put(27.5,-1.75){\vector(2,1){1}}
\put(22.5,1.75){\vector(2,1){1}}
\put(22.5,-1.75){\vector(2,-1){1}}

\put(5,0){\circle*{2}}
\put(15,0){\circle*{2}}
\put(25,2.5){\circle*{2}}
\put(25,-2.5){\circle*{2}}
\put(20.7,-0.7){\line(1,-1){4}}
\put(20.7,-0.7){\vector(1,-1){3}}
\put(30.6,-0.8){\line(1,-1){4}}
\put(30.8,-0.6){\line(2,-1){9}}
\put(30.9,-0.4){\line(3,-1){14}}
\put(30.6,-0.8){\vector(1,-1){2.5}}
\put(30.8,-0.6){\vector(2,-1){5}}
\put(30.9,-0.4){\vector(3,-1){7.5}}
\put(25,-5){\circle*{2}}
\put(35,-5){\circle*{2}}
\put(40,-5){\circle*{2}}
\put(45,-5){\circle*{2}}
\end{picture}
\end{center}

{\bf Step 3:} Linearly order the vertices of $\tilde{\D}$
extending the order of the vertices of the original labeled floor
diagram $\D$ such that, as before, each edge is directed from a
smaller vertex to a larger vertex.

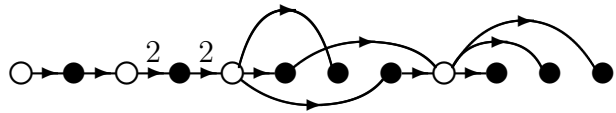
\begin{center}
\begin{picture}(50,40)(65,-15)\setlength{\unitlength}{4pt}\thicklines
\put(12.5,2){\makebox(0,0){$2$}}
\put(17.5,2){\makebox(0,0){$2$}}
\multiput(0,0)(10,0){3}{\circle{2}}
\multiput(40,0)(10,0){1}{\circle{2}}
\multiput(5,0)(10,0){5}{\circle*{2}}
\multiput(30,0)(10,0){1}{\circle*{2}}
\multiput(50,0)(5,0){2}{\circle*{2}}
\put(1,0){\line(1,0){8}}
\put(11,0){\line(1,0){8}}
\put(2.5,0){\vector(1,0){1}}
\put(7.5,0){\vector(1,0){1}}
\put(12.5,0){\vector(1,0){1}}
\put(17.5,0){\vector(1,0){1}}
\put(21,0){\line(1,0){3}}
\put(22.5,0){\vector(1,0){1}}
\qbezier(20.6,0.6)(22,6)(25,6)\qbezier(25,6)(28,6)(29.4,0.6)
\put(25,6){\vector(1,0){1}}
\qbezier(20.8,-0.6)(23.75,-3)(27.5,-3)\qbezier(27.5,-3)(32.25,-3)(34.2,-0.6)
\put(27.5,-3){\vector(1,0){1}}
\qbezier(25.8,0.6)(28.75,3)(32.5,3)\qbezier(32.5,3)(37.25,3)(39.2,0.6)
\put(32.5,3){\vector(1,0){1}}
\put(36,0){\line(1,0){3}}
\put(37.5,0){\vector(1,0){1}}
\put(41,0){\line(1,0){3}}
\put(42.5,0){\vector(1,0){1}}
\qbezier(40.8,0.6)(42,3)(45,3)\qbezier(45,3)(48,3)(49.2,0.6)
\put(45,3){\vector(1,0){1}}
\qbezier(40.8,0.6)(43.5,5)(47.5,5)\qbezier(47.5,5)(51.5,5)(54.2,0.6)
\put(47.5,5){\vector(1,0){1}}

\end{picture}
\end{center}

The extended graph $\tilde{\D}$ together with the linear order on
its vertices is called a \emph{marked floor diagram}, or a
\emph{marking} of the original labeled floor diagram $\D$.

\end{definition}
We want to count marked floor diagrams up to equivalence. Two markings
$\tilde{\D}_1$, $\tilde{\D}_2$ of a labeled floor diagram $\D$ are \emph{equivalent} if there exists an 
automorphism of weighted graphs which preserves the vertices of $\D$ and maps $\tilde{\D}_1$ to
$\tilde{\D}_2$. The \emph{number of markings} $\nu(\D)$ is the number of
marked floor diagrams $\tilde{\D}$ up to equivalence.

\begin{example}
The labeled floor diagram $\D$ of Example~\ref{ex:floordiagram} has $\nu(\D) =
7$ markings (up to equivalence): In step 3 the extra $1$-valent vertex connected to the third white vertex
from the left can be inserted in three ways between the third and fourth
white vertex (up to equivalence) and in four ways right of the fourth
white vertex (again up to equivalence).
\end{example}

Now we can make precise how to compute Severi degrees $N^{d, \delta}$ and
Gromov-Witten invariants $N_{d,g}$
in terms of combinatorics of labeled floor diagrams, thereby
reformulating the initial question of this paper. Part
\ref{itm:GWcount} is not needed in the sequel and only included for
completeness. It first appeared in \cite[Theorem 1]{BM1}.

\begin{theorem}
  \cite[Corollary 1.9, Theorem 1.6]{FM}
\label{thm:floorcount}
$\left. \right.$
\begin{enumerate}
\item The Severi degree $N^{d, \delta}$, i.e., the number of possibly
  reducible nodal curves in $\PP^2$ of degree $d$ with $\delta$ nodes
  through $\frac{d(d+3)}{2} - \delta$ generic points,
  is equal to
\[
N^{d,\delta} = \sum_\D \mu(\D) \nu(\D),
\]
where $\D$ runs over all possibly disconnected labeled floor diagrams of degree
$d$ and cogenus $\delta$.
\smallskip
\item \label{itm:GWcount}The Gromov-Witten invariant $N_{d,g}$, i.e., the number of
  irreducible curves in $\PP^2$ of degree $d$ and genus $g$
  through $3d+g-1$ generic points, is equal to
\[
N_{d,g} = \sum_\D \mu(\D) \nu(\D),
\]
where $\D$ runs over all connected labeled floor diagrams of degree
$d$ and genus~$g$.
\end{enumerate}
\end{theorem}

\section{Computing Node Polynomials}
\label{sec:nodepolys}
In this section we give an explicit algorithm that symbolically
computes the node polynomials $N_\delta(d)$, for given
$\delta \ge 1$. (As $N^{d, 0} = 1$ for $d \ge 1$, we put $N_0(d) = 1$.) An implementation of this algorithm was used to prove
Theorem~\ref{thm:nodepolys} and Proposition~\ref{prop:Gthreshold}.
 We mostly follow the notation in
\cite[Section 5]{FM}. First, we rephrase
Theorem~\ref{thm:nodepolys} in
more compact notation. For $\delta \le 8$ one recovers
\cite[Theorem~3.1]{KP}. For $\delta \le 14$ this coincides with the conjectural
formulas of \cite[Remark~2.5]{Go}.

\begin{theorem}
\label{thm:nodepolynomials}
The node polynomials $N_\delta(d)$, for $\delta \le 14$, are given by
the generating function $\sum_{\delta \ge 0}N_\delta(d) x^\delta$ via
the transformation
\[
\sum_{\delta \ge 0}N_\delta(d) x^\delta =
\exp \bigg( \sum_{\delta \ge 1 }Q_\delta(d) x^\delta \bigg ),
\]
where
\begin{equation*}
\small
\begin{split}
Q_1(d) & = 3(d-1)^2, \\
Q_2(d) & =  \tfrac{-3}{2}(d-1)(14d-25),\\
Q_3(d) & =  \tfrac{1}{3} (690 d^2-2364 d + 1899),\\
Q_4(d) & =  \tfrac{1}{4} (-12060 d^2 + 47835 d -45207),\\
Q_5(d) & =  \tfrac{1}{5}(217728 d^2 - 965646 d +1031823),\\
Q_6(d) & =  \tfrac{1}{6}(- 4010328 d^2  + 19451628 d -22907925 ),\\
Q_7(d) & =  \tfrac{1}{7}(74884932 d^2  - 391230216 d + 499072374),\\
Q_8(d) & =  \tfrac{1}{8}( - 1412380980 d^2  + 7860785643 d -10727554959),\\
Q_9(d) & =  \tfrac{1}{9}(26842726680 d^2  - 157836614730 d + 228307435911),\\
Q_{10}(d) & =  \tfrac{1}{10}(- 513240952752 d^2  + 3167809665372 d -4822190211285),\\
Q_{11}(d) & =  \tfrac{1}{11}(9861407170992 d^2  - 63560584231524 d + 101248067530602),\\
Q_{12}(d) & =  \tfrac{1}{12}(- 190244562607008 d^2  + 1275088266948600 d -2115732543025293),\\
Q_{13}(d) & =  \tfrac{1}{13}( 3682665360521280 d^2  -
25576895657724768 d + 44039919476860362),\\
Q_{14}(d) & =  \tfrac{1}{14}(-71494333556133600 d^2  +
513017995615177680 d - 913759995239314452).
\end{split}
\end{equation*}
In particular, all $Q_\delta(d)$, for $1 \le \delta \le 14$, are
quadratic in $d$.
\end{theorem}

L.~G\"ottsche \cite{Go}
conjectured that all $Q_\delta(d)$ are quadratic. This
theorem proves his conjecture for $\delta \le
14$.

The basic idea of the algorithm (see \cite[Section 5]{FM}) is to
decompose labeled floor diagrams into
smaller building blocks. These gadgets will be crucial in the proofs of
all theorems in this paper.

\begin{definition}
\label{def:template}
A \emph{template} $\Gamma$ is a directed graph (with
possibly multiple edges) on vertices
$\{0, \dots, l\}$, for $l\ge 1$, and edge
weights $w(e) \in \ZZ_{>0}$, satisfying:
\begin{enumerate}
\item If $i \to j$ is an edge then $i <j$.
\item Every edge $i \stackrel{e}{\to} i+1$ has weight $w(e) \ge
  2$. (No ``short edges.'')
\item For each vertex $j$, $1 \le j \le l-1$, there is an edge
  ``covering'' it, i.e., there exists an edge $i \to k$ with $i <j <k$.
\end{enumerate}
\end{definition}

\begin{figure}[htbp]
\begin{center}
\begin{tabular}{c|c|c|c|c|c|c|c|c}
$\Gamma$ &
$\delta(\Gamma)$ & $\ell(\Gamma)$ 
& $\mu(\Gamma)$ 
& $\varepsilon(\Gamma)$ & $\varkappa(\Gamma)$ 
& $k_{\min}(\Gamma)$ & $P(\Gamma,k)$ & $s(\Gamma)$
\\
\hline \hline
&&&&&&&&\\[-.1in]
\begin{picture}(95,8)(-10,-4)\setlength{\unitlength}{2.5pt}\thicklines
\multiput(0,0)(10,0){2}{\circle{2}}
\put(5,2){\makebox(0,0){$\scriptstyle 2$}}
\Eeee
\end{picture}
& 1 & 1 & 4 & 0 & (2) & 2 & $k-1$ & 1
\\[.15in]
\begin{picture}(95,8)(-10,-4)\setlength{\unitlength}{2.5pt}\thicklines
\ooo
\qbezier(0.8,0.6)(10,5)(19.2,0.6)
\end{picture}
& 1 & 2 & 1 & 1 & (1,1) & 1 & $2k+1$ & 0
\\
\hline \hline
&&&&&&&&\\[-.1in]
\begin{picture}(95,8)(-10,-4)\setlength{\unitlength}{2.5pt}\thicklines
\multiput(0,0)(10,0){2}{\circle{2}}
\put(5,2){\makebox(0,0){$\scriptstyle 3$}}
\Eeee
\end{picture}
& 2 & 1 & 9 & 0 & (3) & 3 & $k-2$ & 1
\\[.15in]
\begin{picture}(95,8)(-10,-4)\setlength{\unitlength}{2.5pt}\thicklines
\multiput(0,0)(10,0){2}{\circle{2}}
\put(5,3.5){\makebox(0,0){$\scriptstyle 2$}}
\put(5,-3.5){\makebox(0,0){$\scriptstyle 2$}}
\qbezier(0.8,0.6)(5,2)(9.2,0.6)
\qbezier(0.8,-0.6)(5,-2)(9.2,-0.6)
\end{picture}
& 2 & 1 & 16 & 0 & (4) & 4 & $\binom{k-2}{2}$ & 2
\\[.15in]
\begin{picture}(95,8)(-10,-4)\setlength{\unitlength}{2.5pt}\thicklines
\ooo
\qbezier(0.8,0.6)(10,4)(19.2,0.6)
\qbezier(0.8,-0.6)(10,-4)(19.2,-0.6)
\end{picture}
& 2 & 2 & 1 & 1 & (2,2) & 2 & $\binom{2k}{2}$ & 0
\\[.15in]
\begin{picture}(95,8)(-10,-4)\setlength{\unitlength}{2.5pt}\thicklines
\ooo
\qbezier(0.8,0.6)(10,4)(19.2,0.6)
\put(5,-2){\makebox(0,0){$\scriptstyle 2$}}
\Eeee
\end{picture}
& 2 & 2 & 4 & 1 & (3,1) & 3 & $2k(k-2)$ & 1
\\[.15in]
\begin{picture}(95,8)(-10,-4)\setlength{\unitlength}{2.5pt}\thicklines
\ooo
\qbezier(0.8,0.6)(10,4)(19.2,0.6)
\put(15,-2){\makebox(0,0){$\scriptstyle 2$}}
\eEee
\end{picture}
& 2 & 2 & 4 & 0 & (1,3) & 2 & $2k(k-1)$ & 1
\\[.15in]
\begin{picture}(95,8)(-10,-4)\setlength{\unitlength}{2.5pt}\thicklines
\oooo
\qbezier(0.8,0.6)(15,6)(29.2,0.6)
\end{picture}
& 2 & 3 & 1 & 1 & (1,1,1) & 1 & $3(k+1)$ & 0
\\[.15in]
\begin{picture}(95,8)(-10,-4)\setlength{\unitlength}{2.5pt}\thicklines
\oooo
\qbezier(0.8,0.6)(10,5)(19.2,0.6)
\qbezier(10.8,0.6)(20,5)(29.2,0.6)
\end{picture}
& 2 & 3 & 1 & 1 & (1,2,1) & 1 & $k(4k+5)$ & 0
\\[-.05in]
\end{tabular}
\end{center}
\caption{The templates with $\delta(\Gamma) \le 2$.}
\label{fig:templates}
\end{figure}

Every template $\Gamma$ comes with some numerical
data associated with it. Its \emph{length} $l(\Gamma)$ is the number of
vertices minus $1$. The product of squares of the edge weights
is its \emph{multiplicity} $\mu(\Gamma)$. Its \emph{cogenus} $\delta(\Gamma)$
is
\[
\delta(\Gamma) = \sum_{i \stackrel{e}{\to} j} \bigg[(j-i) w(e) -1
\bigg].
\]

For $1 \le j \le l(\Gamma)$ let $\varkappa_j = \varkappa_j(\Gamma)$
denote the sum of the weights of edges $i \to k$ with $i
< j \le k$ and define
\[
k_{\min}(\Gamma) = \max_{1 \le j \le l}(\varkappa_j -j +1).
\]
This makes $k_{\min}(\Gamma)$ the smallest positive integer $k$ such
that $\Gamma$ can appear in a floor diagram on $\{1, 2, \dots \}$ with
left-most vertex $k$.
Lastly, set
 \begin{displaymath}
   \eps(\Gamma)  = \left\{
     \begin{array}{ll}
       1 &  \text{if all edges arriving at } l \text{ have weight }1,\\
       0 &  \text{otherwise.}
     \end{array}
   \right.
\end{displaymath}
Figure~\ref{fig:templates} (Figure 10 taken from \cite{FM}) lists all templates $\Gamma$ with
$\delta(\Gamma) \le 2$.

A labeled floor diagram $\D$ with $d$ vertices decomposes into an ordered
collection $(\Gamma_1, \dots, \Gamma_m)$ of templates as
follows: First, add an additional vertex $d+1$ ($> d$) to $\D$ along with, for
every vertex $j$ of $\D$, $1 - div(j)$ new edges of
weight $1$ from $j$ to the new vertex
$d+1$. The resulting floor diagram $\D'$ has divergence $1$
at every vertex coming from $\D$. Now remove all \emph{short edges}
from $\D'$, that is, all edges of weight~$1$ between consecutive
vertices. The result is an ordered collection of templates $(\Gamma_1,
\dots, \Gamma_m)$, listed 
left to right, and it is not hard to see that $\sum \delta(\Gamma_i) = \delta(\D)$. 
This process is
reversible once we record the smallest vertex $k_i$ of each
template $\Gamma_i$ (see Example~\ref{ex:templatedecomposition}).

\begin{example}
\label{ex:templatedecomposition}
An example of the decomposition of a labeled floor diagram into
templates is illustrated below. Here, $k_1 = 2$ and $k_2 = 4$.


\begin{picture}(50,121)(-165,-103)\setlength{\unitlength}{3.0pt}\thicklines
\ooooo\Eeee\eEee\eeEe \eeOe \eeeE 
\put(15,2){\makebox(0,0){$2$}} 
\put(7,0){\vector(1,0){1}} 
\put(17,0){\vector(1,0){1}} 
\put(27,0){\vector(1,0){1}} 
\put(27.5,1.75){\vector(2,-1){1}}
\put(27.5,-1.75){\vector(2,1){1}}
\put(35,2){\makebox(0,0){$3$}} 
\put(37,0){\vector(1,0){1}} 
\put(20,-8){\LARGE$\updownarrow$}
\put(-13,-1){$\D = $}
\end{picture}
\begin{picture}(50,121)(-111,-55)\setlength{\unitlength}{3pt}\thicklines
\oooooo\Eeee\eEee\eeEe \eeOe \eeeE \eeeBB \eeeeE \eeeeO
\put(15,2){\makebox(0,0){$2$}} 
\put(7,0){\vector(1,0){1}} 
\put(17,0){\vector(1,0){1}} 
\put(27,0){\vector(1,0){1}} 
\put(27.5,1.75){\vector(2,-1){1}}
\put(27.5,-1.75){\vector(2,1){1}}
\put(35,2){\makebox(0,0){$3$}} 
\put(37,0){\vector(1,0){1}} 
\put(47,0){\vector(1,0){1}} 
\put(40,5){\vector(1,0){1}} 
\put(47.5,1.75){\vector(2,-1){1}}
\put(47.5,-1.75){\vector(2,1){1}}
\put(20,-8){\LARGE$\updownarrow$}
\put(-13,-1){$\D' = $}
\end{picture}
\begin{picture}(50,121)(-57,-10)\setlength{\unitlength}{3pt}\thicklines
\oooooo \eEee  \eeeE \eeeBB 
\put(-3,-1){\big (}
\put(1,-1){\big )}
\put(15,2){\makebox(0,0){$2$}} 
\put(17,0){\vector(1,0){1}} 
\put(35,2){\makebox(0,0){$3$}} 
\put(37,0){\vector(1,0){1}} 
\put(40,5){\vector(1,0){1}} 
\put(-22,-1){$(\Gamma_1, \Gamma_2) = $}
\end{picture}
\end{example}

To each template $\Gamma$ we associate a polynomial that records the
number of
``markings of $\Gamma$:'' For $k \in \ZZ_{>0}$ let $\Gamma_{(k)}$ denote
the graph obtained from $\Gamma$ by first adding $k+i-1-\varkappa_i$
short edges connecting $i-1$ to i, for $1 \le i \le l(\Gamma)$, and
then subdividing each edge of the resulting graph by introducing one
new vertex for each edge. By \cite[Lemma 5.6]{FM} the number of
linear extensions (up to equivalence) of
the vertex poset of the graph $\Gamma_{(k)}$ extending the vertex
order of $\Gamma$ is a polynomial in $k$, if $k \ge k_{\min}(\Gamma)$,
which we denote by $P(\Gamma,k)$ (see Figure~\ref{fig:templates}). The number of markings of a labeled
floor diagram $\D$ decomposing into templates $(\Gamma_1, \dots,
\Gamma_m)$ is then
\[
\nu(\D) = \prod_{i = 1}^m P(\Gamma_i, k_i),
\]
where $k_i$ is the smallest vertex of $\Gamma_i$ in $\D$. The algorithm is based on

\begin{theorem}[\cite{FM}, (5.13)]
\label{thm:severiformula}
The Severi degree $N^{d,\delta}$, for $d, \delta \ge 1$, is given by the
template decomposition formula
\begin{equation}
\label{eqn:5.13}
\sum_{(\Gamma_1, \dots, \Gamma_m)}
\prod_{i=1}^m \mu(\Gamma_i) \hspace{-2mm}
\sum_{k_m=k_{\min}(\Gamma_m)}^{d-l(\Gamma_m)+\eps(\Gamma_m)}
\hspace{-5mm} P(\Gamma_m, k_m)
 \cdots
\sum_{k_1=k_{\min}(\Gamma_1)}^{k_2-l(\Gamma_1)} \hspace{-3mm} P(\Gamma_1, k_1),
\end{equation}
where the first sum is over all ordered collections of templates
$(\Gamma_1, \dots, \Gamma_m)$, for all $m \ge 1$, with $\sum_{i=1}^m
\delta(\Gamma_i) = \delta$, and the sums indexed by $k_i$, for $1 \le i <
m$, are over $k_{\min}(\Gamma_i) \le k_i \le k_{i+1}-l(\Gamma_i)$,
\end{theorem}

Expression (\ref{eqn:5.13}) can be evaluated symbolically, using the
following two lemmata. The first is Faulhaber's formula \cite{Kn93} from
1631 for discrete
integration of polynomials. The second treats lower limits of iterated discrete
integrals and its proof is straightforward. Here $B_j$ denotes the
$j$th Bernoulli number with the convention that $B_1=+\frac{1}{2}$.
 
\begin{lemma}[\cite{Kn93}]
\label{lem:faulhaber}
Let $f(k) = \sum_{i=0}^d c_i k^i$ be a polynomial in $k$. Then, for $n \ge 0$,
\begin{equation}
\label{eqn:discreteint}
F(n) \stackrel{\text{def}}{=} \sum_{k=0}^{n} f(k) = \sum_{s=0}^d 
\frac{c_s}{s+1} \sum_{j=0}^s{{s+1} \choose{j}} B_j n^{s+1-j}.
\end{equation}
In particular, $deg(F) = deg(f) + 1 $. 
\end{lemma}

\begin{lemma}
\label{lem:lowerlimits}
Let $f(k_1)$ and $g(k_2)$ be polynomials in $k_1$ and $k_2$,
respectively, and let $a_1, b_1, a_2,
b_2 \in \ZZ_{\ge 0}$. Furthermore, let $F(k_2) = \sum_{k_1 =
  a_1}^{k_2 - b_1} f(k_1)$ be a discrete anti-derivative of $f(k_1)$, where $k_2 \ge a_1 + b_1$. Then, for $ n \ge \max(a_1 + b_1 + b_2, a_2 + b_2)$,
\begin{displaymath}
\sum_{k_2 = a_2}^{n-b_2} g(k_2) \sum_{k_1 = a_1}^{k_2 - b_1} f(k_1)
= \sum_{k_2 = \max(a_1 + b_1, a_2)}^{n-b_2} \hspace{-6mm} g(k_2) F(k_2).
\end{displaymath}
\end{lemma}

\begin{example}
\label{ex:lowerlimits}
An illustration of  Lemma~\ref{lem:lowerlimits} is the following iterated discrete integral:
\begin{displaymath}
\sum_{k_2 = 1}^n \sum_{k_1 = 1}^{k_2-2} 1 = \sum_{k_2 = 1}^n
\left\{
     \begin{array}{ll}
       k_2 -2 &  \text{if } k_2 \ge 2\\
       0 &   \text{if } k_2 = 1
     \end{array}
   \right\}
= \sum_{k_2 = 3}^n \big( k_2-2 \big).
\end{displaymath}
\end{example}

\begin{algorithm}[tb]
\label{alg:nodepoly}
\SetAlgoLined 
\KwData{The cogenus $\delta$.}
\KwResult{The node polynomial $N_\delta(d)$.}
\Begin{
Generate all templates $\Gamma$ with $\delta(\Gamma) \le \delta$\;
$N_\delta(d) \leftarrow 0$\;
\ForAll{ordered collections of templates $\tilde{\Gamma} =(\Gamma_1, \dots,
\Gamma_m)$ with $\sum_{i=1}^m \delta(\Gamma_i) = \delta$}{
$i \leftarrow 1$\; $Q_1 \leftarrow 1$\;
\While{$i \le m$}{
$a_i \leftarrow \max \big( k_{\min}(\Gamma_i), k_{\min}(\Gamma_{i-1}) +
l(\Gamma_{i-1}), \dots, k_{\min}(\Gamma_1) + l(\Gamma_1) + \cdots +
l(\Gamma_{i-1}) \big)$;
}
\While{$i \le m-1$}{
$Q_{i+1}(k_{i+1}) \leftarrow \sum_{k_i = a_i}^{k_{i+1}-l(\Gamma_i)}
P(\Gamma_i, k_i) Q_i(k_i)$\; 
$i \leftarrow i+1$;
}
$Q^{\tilde{\Gamma}}(d) \leftarrow \sum_{k_m = a_m}^{d-l(\Gamma_m)+\eps(\Gamma_m)}
P(\Gamma_m, k_m) Q_m(k_m)$\;
$Q^{\tilde{\Gamma}}(d) \leftarrow \prod_{i=1}^m \mu(\Gamma_i) \cdot
Q^{\tilde{\Gamma}}(d)$\;
$N_\delta(d) \leftarrow N_\delta(d) + Q^{\tilde{\Gamma}}(d)$\;
}
}
\caption{Algorithm to compute node polynomials.}
\end{algorithm}

Using these results Algorithm \ref{alg:nodepoly} computes node polynomials $N_\delta(d)$
for an arbitrary number of nodes $\delta$. The first step, the
template generation, is explained later in this section.

\begin{proof}[Proof of Correctness of Algorithm~\ref{alg:nodepoly}.]
The algorithm is a direct implementation of Theorem~\ref{thm:severiformula}. The \mbox{$m$-fold} discrete integral is evaluated symbolically,
one sum at a time, using Faulhaber's formula (Lemma~\ref{lem:faulhaber}). The lower limit $a_i$ of the $i$th sum is
given by an iterated application of Lemma~\ref{lem:lowerlimits}.
\end{proof}

As Algorithm~\ref{alg:nodepoly} is stated its termination in
reasonable time is hopeless for $\delta \ge 8$ or $9$. The novelty of this section, together with
an explicit formulation, is how to implement
the algorithm efficiently. This is explained in Remark~\ref{rmk:improvements}.

\begin{remark}
\label{rmk:improvements}
The running time of the algorithm can be
improved vastly as follows: As the limits of summation in
(\ref{eqn:5.13}) only depend on $k_{\min}(\Gamma_i)$,
$l(\Gamma_i)$ and $\eps(\Gamma_m)$, we can replace the template
polynomials $P(\Gamma_i,k_i)$ by $\sum P(\Gamma_i, k_i)$, where
the sum is over all templates $\Gamma_i$ with prescribed
$(k_{\min}, l, \eps)$. After this transformation the first sum in 
(\ref{eqn:5.13}) is over all combinations of those tuples. This
reduces the computation drastically as, for example, the $167885753$
templates of cogenus $14$ make up only $343$ equivalence
classes. Also, in  (\ref{eqn:5.13}) we can distribute the template multiplicities
$\mu(\Gamma_i)$ and replace $P(\Gamma_i, k_i)$ by $\mu(\Gamma_i)
P(\Gamma_i, k_i)$ and thereby eliminate $\prod \mu(\Gamma_i)$. Another
speed-up is to compute all discrete integrals of monomials using Lemma~\ref{lem:faulhaber} in advance.
\end{remark}

The generation of the templates is the bottleneck of the algorithm. Their number grows rapidly with $\delta$ as can be
seen from Figure~\ref{fig:templatetable}. However, their generation can be
parallelized easily (see below).

Algorithm~\ref{alg:nodepoly} has been implemented in Maple. Computing $N_{14}(d)$ on a machine with two quad-core Intel(R) Xeon(R)
CPU L5420 @ 2.50GHz, 6144 KB cache, and 24 GB RAM took about 70 days.

\begin{remark}
\label{rmk:relativeremark}
Using the combinatorial framework of
floor diagrams one
can show that also \emph{relative Severi degrees} (i.e., the degrees
of \emph{generalized Severi varieties}, see~\cite{CH98,
  Va00}) are polynomial and given by ``relative node polynomials''~\cite[Theorem~1.1]{FB10}. This suggests the existence of a generalization
of G\"ottsche's Conjecture~\cite[Conjecture~2.1]{Go} and the G\"ottsche-Yau-Zaslow
formula~\cite[Conjecture~2.1]{Go}. Thus, the combinatorics of floor
diagrams lead to new conjectures although the techniques and results seem to be
out of reach at this time. 
\end{remark}

\begin{remark}
\label{rmk:numericalalgo}
We can use Algorithm~\ref{alg:nodepoly} to compute the
values of the Severi degrees $N^{d, \delta}$ for prescribed values of 
$d$ and $\delta$. After we specify a degree $d$ and a number of nodes
$\delta$ all sums in our algorithm become finite and can be evaluated
numerically. See Appendix \ref{app:smallSevdegs} for all values of $N^{d,
\delta}$ for $0 \le \delta \le 14$ and $1 \le d \le 13$.
\end{remark}

\begin{proof}[Proof of Proposition~\ref{prop:Gthreshold}]
For $1 \le \delta \le 14$ we observe, using the data in Appendices
\ref{app:newpolys} and \ref{app:smallSevdegs}, that $N_{\delta}(d) =
N^{d, \delta}$ for all $d_0(\delta) \le d < \delta$, where
$d_0(\delta) = \left \lceil\frac{\delta}{2} \right \rceil + 1$ is
G\"ottsche's threshold. Furthermore, $N_{\delta}(d_0(\delta) - 1) \neq
N^{d_0(\delta) - 1, \delta}$ for all $3 \le \delta \le 14$.
\end{proof}

\subsection*{Template Generation}

\begin{figure}[bp]
\label{fig:templatesliding}
\begin{picture}(-55,220)(180,-220) \setlength{\unitlength}{0.75pt}
\xymatrix{
&&
\begin{picture}(55,30)(0,-4)\setlength{\unitlength}{1.7pt}\thicklines
\ooo\Eeee\EEee\BBee
\put(5,-3){\makebox(0,0){$2$}} 
\put(7,0){\vector(1,0){1}} 
\put(10,5){\vector(1,0){1}} 
\put(10,10){\vector(1,0){1}} 
\put(-17,-1){$\Gamma_0 = $}
\end{picture}
 \ar[dl] \ar[dr]
&& \\
&
\begin{picture}(55,30)(10,-4)\setlength{\unitlength}{1.7pt}\thicklines
\oooo\Eeee\EEee\eBBe
\put(5,-3){\makebox(0,0){$2$}} 
\put(7,0){\vector(1,0){1}} 
\put(10,5){\vector(1,0){1}} 
\put(20,10){\vector(1,0){1}} 
\end{picture}
 \ar[dl] \ar[d]
& &
\begin{picture}(55,30)(10,-4)\setlength{\unitlength}{1.7pt}\thicklines
\ooo\eEee\EEee\BBee
\put(15,-3){\makebox(0,0){$2$}} 
\put(17,0){\vector(1,0){1}} 
\put(10,5){\vector(1,0){1}} 
\put(10,10){\vector(1,0){1}} 
\end{picture}
 \ar[dl] \ar[d]
&
\\
\begin{picture}(55,30)(30,-4)\setlength{\unitlength}{1.7pt}\thicklines
\ooooo\Eeee\EEee\eeBB
\put(5,-3){\makebox(0,0){$2$}} 
\put(7,0){\vector(1,0){1}} 
\put(10,5){\vector(1,0){1}} 
\put(30,10){\vector(1,0){1}} 
\put(0, -3){\line(3,1){40}}
\put(0, 10){\line(3,-1){40}}
\end{picture}
&
\begin{picture}(55,30)(20,-4)\setlength{\unitlength}{1.7pt}\thicklines
\oooo\Eeee\eEEe\eBBe
\put(5,-3){\makebox(0,0){$2$}} 
\put(7,0){\vector(1,0){1}} 
\put(20,5){\vector(1,0){1}} 
\put(20,10){\vector(1,0){1}} 
\put(0, -3){\line(3,1){40}}
\put(0, 10){\line(3,-1){40}}
\end{picture}
 & 
\begin{picture}(55,30)(0,-4)\setlength{\unitlength}{1.7pt}\thicklines
\oooo\eEee\EEee\eBBe
\put(15,-3){\makebox(0,0){$2$}} 
\put(17,0){\vector(1,0){1}} 
\put(10,5){\vector(1,0){1}} 
\put(20,10){\vector(1,0){1}} 
\end{picture}
 \ar[dl] & 
\begin{picture}(0,0)(11,-4)
\Bigg(
\end{picture}
 \begin{picture}(55,30)(0,-4)\setlength{\unitlength}{1.7pt}\thicklines
\oooo\eeEe\EEee\BBee
\put(25,-3){\makebox(0,0){$2$}} 
\put(27,0){\vector(1,0){1}} 
\put(10,5){\vector(1,0){1}} 
\put(10,10){\vector(1,0){1}} 
\end{picture} 
\begin{picture}(0,0)(-11,-5)
\Bigg)
\end{picture}
 \ar[dl] \ar[d]
\\
&
\begin{picture}(55,30)(30,-4)\setlength{\unitlength}{1.7pt}\thicklines
\ooooo\eEee\EEee\eeBB
\put(15,-3){\makebox(0,0){$2$}} 
\put(17,0){\vector(1,0){1}} 
\put(10,5){\vector(1,0){1}} 
\put(30,10){\vector(1,0){1}} 
\put(0, -3){\line(3,1){40}}
\put(0, 10){\line(3,-1){40}}
\end{picture}
 & 
\begin{picture}(55,30)(10,-4)\setlength{\unitlength}{1.7pt}\thicklines
\oooo\eeEe\EEee\eBBe
\put(25,-3){\makebox(0,0){$2$}} 
\put(27,0){\vector(1,0){1}} 
\put(10,5){\vector(1,0){1}} 
\put(20,10){\vector(1,0){1}} 
\end{picture}
 \ar[d]   
&
\begin{picture}(55,30)(10,-4)\setlength{\unitlength}{1.7pt}\thicklines
\oooo\eeeE\EEee\BBee
\put(35,-3){\makebox(0,0){$2$}} 
\put(37,0){\vector(1,0){1}} 
\put(10,5){\vector(1,0){1}} 
\put(10,10){\vector(1,0){1}} 
\put(0, -3){\line(3,1){40}}
\put(0, 10){\line(3,-1){40}}
\end{picture}
\\ 
& & 
\begin{picture}(55,30)(30,-4)\setlength{\unitlength}{1.7pt}\thicklines
\ooooo\eeEe\EEee\eeBB
\put(25,-3){\makebox(0,0){$2$}} 
\put(27,0){\vector(1,0){1}} 
\put(10,5){\vector(1,0){1}} 
\put(30,10){\vector(1,0){1}} 
\put(0, -3){\line(3,1){40}}
\put(0, 10){\line(3,-1){40}}
\end{picture}
& \\
} 
\end{picture}
\caption[abc]{Branch-and-bound tree for 
$\alpha =
\begin{bmatrix}
0 & 1  \\
2 & 0 
\end{bmatrix}$.
}
\end{figure}
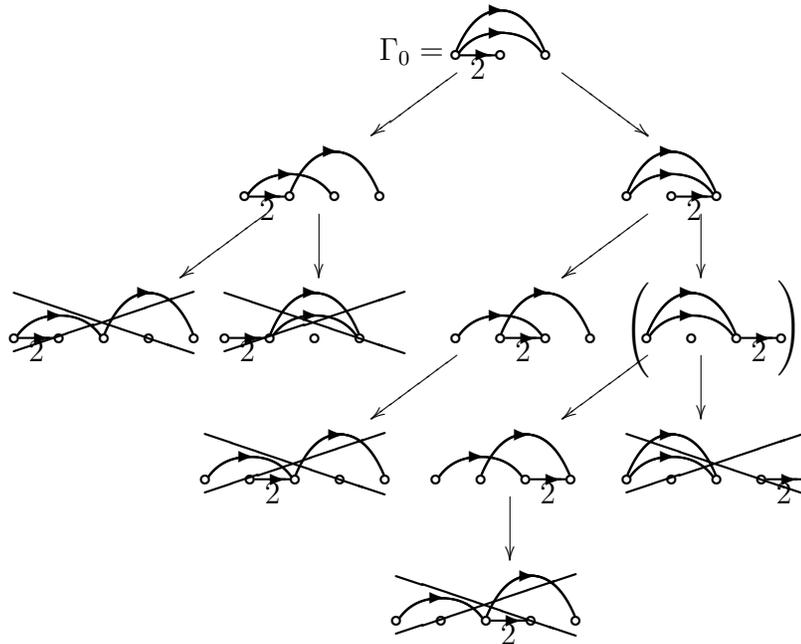

To compute a list of all templates of a given cogenus one can proceed as follows. First,
we need some terminology and notation. An edge $i \to j$
of a template is said to have \emph{length} $j-i$.
A template $\Gamma$ is of \emph{type} $\alpha =(\alpha_{ij})$, $i, j \in \ZZ_{>0}$, if
$\Gamma$ has $\alpha_{ij}$ edges of length $i$ and weight $j$. Every
type $\alpha$ satisfies, by definition of cogenus of a template,
\begin{equation}
\label{eqn:type}
\sum_{i, j \ge 1} \alpha_{ij} (i\cdot j -1) = \delta(\Gamma).
\end{equation}
Note that $\alpha_{11} = 0$ as short edges are not allowed in
templates. The number of types constituting a given cogenus $\delta$ is finite.

\begin{algorithm}[tb]
\label{alg:templategeneration}
\SetAlgoLined
\KwData{A graph $A$ with a distinguished edge $e_1$.}
\KwResult{An infinite directed tree of graphs with root $A$.}
\Begin{
\ForAll{edges $e_2$ of $A$ with $e_2 \ge e_1$ (in the fixed order)}{
$B \leftarrow$ graph obtained from $A$ by moving $e_2$ to the next
vertex\;
\If{the natural partial order (from left to right) of the edges of $B$
  that are of the
  same type as $e_2$ is compatible with the fixed order}{
Insert $B$ as a child of $A$\;

Execute this procedure with input $(B, e_2)$\;
}
}
}
\caption{A recursion which can generates a tree containing all
  templates of a given type.}
\end{algorithm}

We can generate all templates of type $\alpha$ using a branch-and-bound algorithm which slides edges in a suitable order. Let $\Gamma_0$ be the unique template
of type $\alpha$ with all edges emerging from vertex $0$.
Fix a linear order on the set of edges
of type $\alpha$. For example, if $\alpha~ =~\begin{bmatrix}
0 & 1  \\
2 & 0  \end{bmatrix}$, we could choose: 
\begin{equation*}
\begin{picture}(75,25)(-20,-10)\setlength{\unitlength}{2.5pt}\thicklines
\oo\Eeee
\put(5,-3){\makebox(0,0){$2$}} 
\put(7,0){\vector(1,0){1}} 
\end{picture}
\begin{picture}(100,25)(-20,-10)\setlength{\unitlength}{2.5pt}\thicklines
\ooo\EEee
\put(10,5){\vector(1,0){1}} 
\put(-12,-1){\LARGE$< $}
\end{picture}
\begin{picture}(100,25)(-20,-10)\setlength{\unitlength}{2.5pt}\thicklines
\ooo \BBee
\put(10,10){\vector(1,0){1}} 
\put(-12,-1){\LARGE$< $}
\put(22,-1){.}
\end{picture}
\end{equation*}

Algorithm
\ref{alg:templategeneration} applied to the pair $(\Gamma_0, e_0)$, where $e_0$
is the smallest edge of
$\Gamma_0$, creates an infinite directed tree with root
$\Gamma_0$ all of whose vertices
correspond to different graphs. Eliminate a branch if either
\begin{enumerate}
\item no edge of the root of the branch starts at vertex $1$, or
\item condition $(3)$ in Definition~\ref{def:template} is impossible to
satisfy for graphs further down the tree.
\end{enumerate}
See  Figure~\ref{fig:templatesliding} for an illustration
for $\alpha~ =~\begin{bmatrix}
0 & 1  \\
2 & 0  \end{bmatrix}$. 

A complete, non-redundant list of all templates of type $\alpha$ is
then given by all remaining graphs which satisfy condition (3) of
Definition~\ref{def:template} as every template can be obtained in a
unique way from $\Gamma_0$ by shifting edges in an order that is
compatible with the order fixed earlier.
Note that it can happen that a non-template graph
precedes a template within a branch. For an example see the graph in
brackets in Figure~\ref{fig:templatesliding}. Template generation for
different types can be executed in parallel. The number of templates, for $\delta \le 14$,
is given in Figure~\ref{fig:templatetable}.

\begin{figure}[htb]
\begin{equation*}
\begin{array}{r|r||r|r||r|r}
\delta & \# \text{ of templates} & \delta & \# \text{ of templates} &\delta & \# \text{ of templates} \\
\hline
1 & 2 & 6 & 1711 & 11 & 2233572 \\
2 & 7 & 7 & 7135 & 12 & 9423100\\
3 & 26 &  8 &29913 & 13 & 39769731\\
4 & 102 &  9 & 125775 & 14 & 167885753\\
5 & 414 &  10 & 529755 &&\\
\end{array}
\end{equation*}
\caption{The number of templates with cogenus $\delta \le 14$.}
\label{fig:templatetable}
\end{figure}

\section{Threshold Values}
\label{sec:thresholdvalues}

S.~Fomin and G.~Mikhalkin \cite[Theorem 5.1]{FM} proved polynomiality
of the
Severi degrees $N^{d,\delta}$ in $d$, for fixed $\delta$, provided $d$ is
sufficiently large. More
precisely, they showed that $N_\delta(d) = N^{d, \delta}$ for $d \ge
2\delta$. In this section we show that
their threshold can be improved to $d \ge \delta$ (Theorem~\ref{thm:threshold}).

We need the following elementary observation about robustness of discrete
anti-derivatives of polynomials whose continuous counterpart is the
well known fact that $\int_{a-1}^{a-s- 1} f(x) dx = 0$ if $f(x) = 0$
on the interval $(a-s-1, a-1)$.

\begin{lemma}
\label{lem:discretelem}
For a polynomial $f(k)$ and $a \in \ZZ_{> 0}$ let $F(n) =
\sum_{k=a}^n f(k)$ be the polynomial in $n$ uniquely determined
by large enough values of $n$. ($F(n)$ is a polynomial by
Lemma~\ref{lem:faulhaber}.) If we have $f(a-1) =
\cdots = f(a-s) = 0$ for some $0 \le s < a$ (this
condition is vacuous for $s = 0$) then it also holds that
$F(a-1) = \cdots = F(a - s - 1) = 0$.
In particular,
$\sum_{k=a}^n f(k)$ is a polynomial in $n$, for $n \ge a- s-1$.
\end{lemma}

Even for $s = 0$ the lemma is non-trivial as, in general, $F(a-2) \neq
0$.

\begin{proof}
Let $G(n)$ be the polynomial in $n$ defined via $G(n)= \sum_{k=0}^n f(k)$ for
large $n$. Then $F(n) = G(n) - \sum_{k=0}^{a-1}f(k)$ for all
$n \in \ZZ_{\ge 0}$. In particular, for any $0 \le i \le s$, we have
$F(a-i -1) = G(a-i-1) -  \sum_{k=0}^{a-1} f(k)= G(a-i-1) -
\sum_{k=0}^{a-i-1} f(k) = 0$.
\end{proof}

Recall that for a template $\Gamma$, we defined $k_{\min} = k_{\min}(\Gamma)$ to
be the smallest $k \ge 1$ such that $k + j -1 \ge \varkappa_j(\Gamma)$
for all $1 \le j \le l(\Gamma)$. Let $j_0$ be the smallest $j$ for
which equality is attained (it is easy to see that equality is
attained for some $j$). Define $s(\Gamma)$ to be the number of edges
of $\Gamma$ from $j_0-1$ to $j_0$ (of any weight). See Figure~\ref{fig:templates} for
some examples. The following lemma shows that the template polynomials
$P(\Gamma, k)$ satisfy the condition of Lemma~\ref{lem:discretelem}.

\begin{lemma}
\label{lem:templatepolyevaluations}
With the notation from above it holds that
\[
P(\Gamma, k_{\min} - 1) = P(\Gamma, k_{\min} - 2) = \cdots = P(\Gamma, k_{\min} - s(\Gamma)) = 0.
\]
\end{lemma}

\begin{proof}
Recall from Section~\ref{sec:nodepolys} that, for $k \ge k_{\min}(\Gamma)$, the polynomial $P(\Gamma,
k)$ records the number of linear extension (up to equivalence) of some poset
$\Gamma_{(k)}$ which is obtained from $\Gamma$ by first adding
$k+j-1-\varkappa_j(\Gamma)$ ``short edges'' connecting $j - 1$ to $j$, for $1
\le j \le l(\Gamma)$, and then subdividing each edge of the resulting
graph by introducing a new vertex for each edge.

Using the notation from the last paragraph notice that $k_{\min} + j_0
- 1 = \varkappa_{j_0}(\Gamma)$, and thus $\Gamma_{(k)}$ has $k-k_{\min}$ ``short
edges'' between $j_0 - 1$ and $j_0$. Every linear extension of $\Gamma_{(k)}$ can be obtained
by first linearly ordering the midpoints of these $k-k_{\min}$ ``short
edges'' and the midpoints of the $s(\Gamma)$ many edges of $\Gamma$
connecting $j_0 - 1$ and $j_0$ before completing the linear order to
all vertices of  $\Gamma_{(k)}$. Therefore, the polynomial $(k-k_{\min} + 1) \cdots (k -
k_{\min} + s(\Gamma))$ divides $P(\Gamma, k)$.
\end{proof}

Before we can prove Theorem~\ref{thm:threshold} we need a last
technical lemma.

\begin{lemma}
\label{lem:technicalinequality}
Using the notation from above we have, for each template $\Gamma$,
\[
k_{\min}(\Gamma) - s(\Gamma) + l(\Gamma) - \varepsilon(\Gamma) \le \delta(\Gamma) + 1.
\]
\end{lemma}

\begin{proof}
As before, let $j_0$ be the smallest $j$ in $\{1, \dots, l(\Gamma) \}$
with $k_{\min} + j - 1 = \varkappa_j(\Gamma)$. It suffices to show
that $
\varkappa_{j_0}(\Gamma) - j_0 - s(\Gamma) + l(\Gamma) - \eps(\Gamma) \le
\delta(\Gamma) $.

Let $\Gamma'$ be the template obtained from $\Gamma$ by removing all
edges $i \to k$ with either $k < j_0$ or $i \ge j_0$. It is
easy to see that $l(\Gamma) - \eps(\Gamma)
- (l(\Gamma') - \eps(\Gamma')) \le \delta(\Gamma) -
\delta(\Gamma')$. Thus,  we can assume without
loss of generality that all edges $i \to k$ of $\Gamma$ satisfy $i <
j_0 \le k$. Therefore, as $\varkappa_{j_0}(\Gamma) =
\sum_{
\text{ edges } e \text{ of }\Gamma
} \wt(e)$ it suffices to
show that
\begin{equation}
\label{eqn:technicalinequality1}
l(\Gamma) - \eps(\Gamma) \le \sum_{\text{edges }e \text{ of } \Gamma} \Big[ \wt(e) (\len(e) -
1)  - 1 \Big] + s(\Gamma) + j_0,
\end{equation}
where $\len(e)$ is the length of an edge $e$. The contribution of the $s(\Gamma)$
edges of $\Gamma$ between $j_0 - 1$  and $j_0$ to the
sum is $-s(\Gamma)$, thus the right-hand-side of (\ref{eqn:technicalinequality1})
equals
\begin{equation}
\label{eqn:technicalinequality2}
\sum \Big[ \wt(e) (\len(e) -
1) - 1 \Big]  + j_0 
\end{equation}
with the sum now running over all edges of $\Gamma$ of length at least
2. If there are no such edges, then $l(\Gamma) = 1$ and we are
done. Otherwise, if $\eps(\Gamma) = 1$, expression
(\ref{eqn:technicalinequality2}) equals $\sum (\len(e) - 2) + j_0$,
which is $\ge l(\Gamma) - 2 + j_0$ or $\ge  l(\Gamma) - 3 + j_0$ if $j_0 \in \{1,
l(\Gamma) \}$ or
$1 < j_0 < l(\Gamma)$, respectively (by considering only edges adjacent
to vertices $0$ and $l(\Gamma)$ of $\Gamma$). In either case the result
follows.

If $\eps(\Gamma) = 0$ then expression
(\ref{eqn:technicalinequality2}) is $\ge l(\Gamma) + (l(\Gamma) - 3 +
j_0)$ or $\ge l(\Gamma) -2 + j_0$ if $j_0 \in \{ 1, l(\Gamma) \}$ or
$1 < j_0 < l(\Gamma)$, respectively. This completes the proof.
\end{proof}

\begin{proof}[Proof of Theorem~\ref{thm:threshold}]
By Lemma~\ref{lem:lowerlimits} and repeated application of
Lemmata~\ref{lem:discretelem} and~\ref{lem:templatepolyevaluations} it suffices to show that $d \ge \delta$ simultaneously implies
\begin{equation}
\label{eq:manyinequal}
\begin{split}
d \ge & \, l(\Gamma_m) - \eps(\Gamma_m) +k_{\min}(\Gamma_m) - s(\Gamma_m)-1, \\
d \ge & \, l(\Gamma_m) - \eps(\Gamma_m) + l(\Gamma_{m-1}) +
      k_{\min}(\Gamma_{m-1}) - s(\Gamma_{m-1})- 2, \\
& \quad \quad \quad \quad \quad \quad \vdots \\
d \ge & \, l(\Gamma_m) - \eps(\Gamma_m) + l(\Gamma_{m-1}) + \cdots +
l(\Gamma_1) + k_{\min}(\Gamma_1) - s(\Gamma_1)- m,
\end{split}
\end{equation}
for all collections of templates $(\Gamma_1, \dots, \Gamma_m)$ with
$\sum_{i =1}^m \delta(\Gamma_i) = \delta$.

The first inequality is a direct consequence of Lemma~\ref{lem:technicalinequality}.
For the other inequalities, notice that $l(\Gamma) - \eps(\Gamma) \le
\delta(\Gamma)$ for all templates $\Gamma$, hence
\begin{displaymath}
l(\Gamma_m) - \eps(\Gamma_m) -1 \le
\delta(\Gamma_m) -1
\end{displaymath}
and 
\begin{displaymath}
l(\Gamma_i) - 1 \le \delta(\Gamma_i), \quad \text{for } 2 \le i \le
m-1.
\end{displaymath}
By Lemma~\ref{lem:technicalinequality} we have
\begin{displaymath}
l(\Gamma_1) + k_{\min}(\Gamma_1) - s(\Gamma_1)- 1 \le
\delta(\Gamma_1) +1
\end{displaymath}
as $\varepsilon(\Gamma_1) \le 1$, and
the right-hand-side of the last inequality of (\ref{eq:manyinequal}) is $\le
\sum_{i=1}^m\delta(\Gamma_i) = \delta \le d$. The proof of the other
inequalities is very similar.
\end{proof}

\section{Coefficients of Node Polynomials}
\label{sec:coefficients}

The goal of this section is to present an algorithm for the
computation of the coefficients of $N_\delta(d)$, for general
$\delta$. The algorithm can be used to prove Theorem~\ref{thm:firstcoeffs} and
thereby confirm and extend a conjecture of P.~Di Francesco and
C.~Itzykson in \cite{DI}
where they conjectured
the $7$ terms of $N_\delta(d)$ of largest degree.

Our algorithm should be able to find formulas for arbitrarily many coefficients
of $N_\delta(d)$. We prove correctness of our algorithm in this
section. The algorithm rests on the polynomiality of solutions of
certain polynomial difference equations (see (\ref{eqn:localrec})).

First, we fix
some notation building on terminology of Section~\ref{sec:nodepolys}. By Remark~\ref{rmk:improvements} we can
replace the polynomials $P(\Gamma,k)$ in (\ref{eqn:5.13}) by
the product
$\mu(\Gamma) P(\Gamma,k)$, thereby removing the product $\prod \mu(\Gamma_i)$ of the
template multiplicities. In this section we write $P^*(\Gamma,k)$ for
$\mu(\Gamma) P(\Gamma,k)$.
For integers $i \ge 0$ and $a \ge 0$ let $M_i(a)$
denote the matrix of the linear map
\begin{equation}
\label{eqn:linmap}
f(k) \mapsto \sum_{\Gamma: \delta(\Gamma) = i} \sum_{k=k_{\min}(\Gamma)}^{n-l(\Gamma)}  P^*(\Gamma, k) \cdot f(k),
\end{equation}
where $f(k) = c_0 k^a + c_1 k^{a-1} + \cdots$, a polynomial of degree
$a$, is mapped to the polynomial $M_i(a)(f(k)) = d_0 n^{a+i +1} + d_1 n^{a+i} +
\cdots$ in $n$. (By Lemma~\ref{lem:faulhaber} and the proof of
Lemma~\ref{lem:independence} the image has degree $a + i + 1$.) Hence
$M_i(a) {\bf c} = {\bf d}$. Similarly, define
$M_i^{\text{end}}(a)$ to be the matrix of the linear map 
\begin{equation}
\label{eqn:linmap2}
f(k) \mapsto \sum_{\Gamma: \delta(\Gamma) = i} \sum_{k=k_{\min}(\Gamma)}^{n-l(\Gamma)+\eps(\Gamma)}  P^*(\Gamma, k) \cdot f(k).
\end{equation}

Later we will consider
square sub-matrices of $M_i(a)$ and $M_i^{\text{end}}(a)$ by
restriction to the first few rows
and columns which will be
denoted $M_i(a)$ and $M_i^{\text{end}}(a)$ as well. Note that $M_i(a)$
and $M_i^{\text{end}}(a)$
are lower triangular. For example, for $a$ large enough,
\begin{displaymath}
M_1(a) = 
\begin{bmatrix}
\frac{6}{a+2} & 0 & 0 & 0 & 0 & \cdots \\
-\frac{5a+8}{a+1} & \frac{6}{a+1} & 0 & 0 & 0 & \cdots\\
\frac{5}{2}a + 3 & -\frac{5a+3}{a} & \frac{6}{a} & 0 & 0 & \cdots \\
-\frac{1}{4}(4a+1)a & \frac{5}{2} a + \frac{1}{2}& -\frac{5a-2}{a-1} & \frac{6}{a-1} & 0 & \cdots\\
\frac{1}{40}(13a^2 - 20a+7)a & -a^2+\frac{7}{4}a -\frac{3}{4} & \frac{5}{2}a-2 & -\frac{5a-7}{a-2} &
\frac{6}{a-2} & \cdots \\
\vdots & \vdots & \vdots & \vdots & \vdots & \ddots
\end{bmatrix}
.
\end{displaymath}


\begin{lemma}
\label{lem:independence}
The first $a + i$ rows of $M_i(a)$ and $M_i^{\text{end}}(a)$ are
independent of the lower limits of summation in (\ref{eqn:linmap}) and
(\ref{eqn:linmap2}), respectively.
\end{lemma}

\begin{proof}
It is an easy consequence of the proof
of \cite[Lemma 5.7]{FM} that the polynomial $P^*(\Gamma,k)$ associated with a template $\Gamma$ has
degree $\le \delta(\Gamma)$. Equality is attained by the
template $\Gamma$ on vertices $0, 1, 2$ with $i$ edges connecting $0$ and $2$
(so $\delta(\Gamma) = i$). As discrete integration
of a polynomial increases the degree by $1$ the polynomial on the
right-hand-side of (\ref{eqn:linmap}) has degree $1 + i +
a$.
\end{proof}

The basic idea of the algorithm is that templates
with higher cogenera do not contribute to higher degree terms of
the node polynomial. With this in mind we define, for each finite
collection $(\Gamma_1, \dots, \Gamma_m)$ of templates, its \emph{type}
$\tau = (\tau_2, \tau_3, \dots)$, where $\tau_i$ is the number
of templates in $(\Gamma_1, \dots, \Gamma_m)$ with cogenus equal to
$i$, for $i \ge 2$. Note that we do not record the number of templates
with cogenus equal to $1$. 

To collect the contributions of all collections of templates
with a given type $\tau$, let $\tau = (\tau_2, \tau_3, \dots)$ and fix $\delta \ge \sum_{j \ge
  2}\tau_j$ (so that there exist template collections $(\Gamma_1,
\dots, \Gamma_m)$ of type
$\tau$ with $\sum \delta(\Gamma_j) = \delta$). We define
two (column) vectors $C_\tau(\delta)$ and
$C_\tau^{\text{end}}(\delta)$ as the coefficient vectors, listed in decreasing order, of the polynomials
\begin{equation}
\label{eqn:Calpha}
\sum_{(\Gamma_1, \dots, \Gamma_m)}
\sum_{k_m=k_{\min}(\Gamma_m)}^{n-l(\Gamma_m)} \hspace{-4mm} P^*(\Gamma_m, k_m)
\cdots
\sum_{k_1=k_{\min}(\Gamma_1)}^{k_2-l(\Gamma_1)} \hspace{-3mm} P^*(\Gamma_1, k_1)
\end{equation}
and 
\begin{equation}
\label{eqn:Calphaend}
\sum_{(\Gamma_1, \dots, \Gamma_m)}
\sum_{k_m=k_{\min}(\Gamma_m)}^{n-l(\Gamma_m)+\eps(\Gamma)} \hspace{-4mm} P^*(\Gamma_m, k_m)
\sum_{k_{m-1}=k_{\min}(\Gamma_{m-1})}^{k_m-l(\Gamma_{m-1})} \hspace{-3mm}\cdots
\sum_{k_1=k_{\min}(\Gamma_1)}^{k_2-l(\Gamma_1)} \hspace{-3mm} P^*(\Gamma_1, k_1)
\end{equation}
in the indeterminate $n$, where the respective first sums are over all ordered collections of templates of
type $\tau$.

It might look like $C_\tau(\delta)$ is a product of some matrices
$M_i(a)$ applied to the polynomial $1$. However, this is not the case. For
example, note that
\begin{displaymath}
C_{(0,0,\dots)}(2) =  
\begin{bmatrix}
\frac{9}{2} \\
-34 \\
88 \\
- \frac{179}{2} \\
30 \\
0 \\
\vdots
\end{bmatrix}
\neq 
\begin{bmatrix}
\frac{9}{2} \\
-34 \\
88 \\
- \frac{179}{2} \\
27 \\
0 \\
\vdots
\end{bmatrix}
=M_1(2) \cdot M_1(0) \cdot
\begin{bmatrix}
1 \\
0 \\
0\\
0 \\
0 \\
0 \\
\vdots
\end{bmatrix}.
\end{displaymath}
This is because, when iterated discrete integrals are evaluated
symbolically, the lower limits
of integration of the outer sums can change depending on the limits of
the inner sums (cf.\ Lemma~\ref{lem:lowerlimits}). This
observation makes it necessary to compute initial values for
recursions (described later) up to a large enough $\delta$. 

Before we can state the main recursion we need two more notations. For
a type $\tau = (\tau_2, \tau_3, \dots)$ and $i \ge 2$ with
$\tau_i >0$ define a new type $\tau\down_i$ via
$(\tau\down_i)_i = \tau_i -1$ and $(\tau\down_i)_j = \tau_j$
for $j \neq i$. Furthermore, let $\defect(\tau) = \sum_{j\ge2}
(j-1)\tau_j$ be the \emph{defect} of $\tau$. The following lemma
justifies this terminology.

\begin{lemma}
\label{lem:degree}
The polynomials (\ref{eqn:Calpha}) and (\ref{eqn:Calphaend}) are of
degree $2 \delta - \defect(\tau)$.
\end{lemma}

\begin{proof}
Let $(\Gamma_1, \dots, \Gamma_m)$ be a collection of templates with
$\sum_{i=1}^m \delta(\Gamma_i) = \delta$ and type~$\tau$. Then, by
applying the argument in the proof of Lemma~\ref{lem:independence}
to each $\Gamma_i$, the polynomials (\ref{eqn:Calpha}) and (\ref{eqn:Calphaend}) have degree
$\delta + m$. The result follows as
\begin{equation*}
\begin{split}
\delta - \defect(\tau) & =\sum_{i=1}^m \delta(\Gamma_i) - \sum_{j\ge
  2}(j-1) \tau_j \\
& = \sum_{i=1}^m \delta(\Gamma_i)  - \sum_{j \ge 2}\left[ \left(
  \sum_{i:\delta(\Gamma_i) = \tau_j} \delta(\Gamma_i) \right) -
\tau_j \right] \\
& = \# \{i: \delta(\Gamma_i) = 1 \} + \sum_{j \ge 2} \tau_j = m.
\end{split}
\end{equation*}

\end{proof}

The last lemma makes precise
which collections of templates contribute to which coefficients of
$N_\delta(d)$. Namely, the first $N$ coefficients of $N_\delta(d)$ of
largest degree
depend only on collections of templates with types $\tau$ such that
$\defect(\tau) < N$. The following recursion is the heart of the algorithm.

\begin{proposition}
\label{prop:rec}
For every type $\tau$ and integer $\delta$ large enough, it holds that
\begin{equation}
\label{eqn:rec}
\begin{split}
C_\tau(\delta) = & \sum_{i: \tau_i \neq 0} M_i \big(2 \delta -i-1 -\defect(\tau)\big) C_{\tau\down i}(\delta-i) \\
& + M_1\big(2\delta - 2 - \defect (\tau)\big) C_\tau(\delta-1).
\end{split}
\end{equation}
More precisely, if we restrict all matrices $M_i$ to be square of size $N-\defect(\tau)$
and all $C_\tau$ to be vectors of length $N-\defect(\tau)$, then
recursion (\ref{eqn:rec})
holds for
\begin{displaymath}
\delta \ge \max \left ( \left
  \lceil\frac{N+1}{2} \right \rceil,
\sum_{j\ge 2}j \tau_j \right ).
\end{displaymath}
\end{proposition}

\begin{proof}
The coefficient vector $C_\tau(\delta)$ is defined by a sum that runs
over all collections of templates $(\Gamma_1, \dots, \Gamma_m)$ of
type $\tau$ (see (\ref{eqn:Calpha})). Partition the set of such collections by
putting $\delta(\Gamma_m) = 1$, or $\delta(\Gamma_m) = 2$,
and so forth. This partitioning splits expression (\ref{eqn:Calpha}) exactly as
in (\ref{eqn:rec}). 

A summand can be written as a product of some matrix $M_i$ and some
vector $C_{\tau\down i}$ if $\delta$ is large enough, namely if $M_i$
does not depend on the lower limits in (\ref{eqn:Calpha}). If we can factor then the polynomials  (\ref{eqn:Calpha}) defining $C_{\tau\down  i}(\delta-i)$ and $C_{\tau}(\delta -1)$ have degrees
\begin{displaymath}
2(\delta -i) - \defect(\tau\down i) =2\delta -2i - \defect(\tau)+
(i-1) = 2\delta - i -1 -\defect(\tau)
\end{displaymath}
by Lemma~\ref{lem:degree} and, similarly, $2\delta -2 -
\defect(\tau)$, respectively. By Lemma~\ref{lem:independence},  if the matrix
$M_i(2\delta-i-1-\defect(\tau))$ is of
size $N - \defect(\tau)$, then it does not depend on the lower limits if and
only if $\delta \ge \frac{N+1}{2}$. In order for $C_\tau(\delta)$ to
be defined (and the above identity to be meaningful) we need to impose
$\delta \ge \sum_{j \ge 2} j \tau_j$.
\end{proof}  

\begin{remark}
\label{rmk:initialdelta}
Later, when we formulate the algorithm, we need to solve
  recursion (\ref{eqn:rec}) together with an initial condition in
  order to obtain an explicit formula for the first
  $N-\defect(\tau)$ entries of
  $C_\tau(\delta)$. It suffices to take
\begin{equation}
\label{eqn:initialdelta}
\delta_0(\tau) \stackrel{\text{def}}{=} \max \left ( \left
  \lceil\frac{N-1}{2} \right \rceil,
\sum_{j\ge 2}j \tau_j \right )
\end{equation}
as for any $\delta > \delta_0(\tau)$  the vector $C_\tau(\delta)$
of length $N-\defect(\tau)$ can be written in terms of matrices
$M_i$ and vectors $C_{\tau'}(\delta')$ for various types $\tau'$
and integers $\delta' < \delta$.
\end{remark}

We propose Algorithm \ref{alg:nodecoeffs} for the computation of the
coefficients of the node polynomial $N_\delta(d)$. We explain how to solve recursion  (\ref{eqn:rec}) below.

\begin{algorithm}[tb]
\label{alg:nodecoeffs}
\SetAlgoLined
\KwData{A positive integer $N$.}
\KwResult{The coefficient vector $C$ of the first $N$ coefficients of $N_\delta(d)$.}
\Begin{
Compute all templates $\Gamma$ with $\delta(\Gamma) \le N$\;

\ForAll{types $\tau$ with $\defect(\tau) < N$}{
Compute initial values $C_\tau(\delta_0(\tau))$ using (\ref{eqn:Calpha}), with
$\delta_0(\tau)$ as in (\ref{eqn:initialdelta})\;

Solve recursion (\ref{eqn:rec}) for first
$N-\defect(\tau)$ coordinates of $C_\tau(\delta)$\;
\text{Set }
 \begin{align*}
\hspace{-40mm}
C_\tau^{\text{end}}(\delta) \leftarrow & \sum_{i: \tau_i \neq 0} M_i^{\text{end}} \big(2 \delta -i-1 -\defect(\tau)\big) C_{\tau\down i}(\delta-i) \\
& + M_1^{\text{end}}\big(2\delta - 2 - \defect (\tau)\big)
C_\tau(\delta-1) \text{\;}
\end{align*}
}
$C \leftarrow 0$\;
\ForAll{types $\tau$ with $\defect(\tau) <N$}{
Shift the entries of $C_\tau^{\text{end}}(\delta)$ down by
$\defect(\tau)$\;
$C \leftarrow C +  \text{ shifted }C_\tau^{\text{end}}(\delta)$\;
}
}
\caption{Computation of the leading coefficients of the node polynomial.}
\end{algorithm}

\begin{proof}[Proof of Correctness of Algorithm~\ref{alg:nodecoeffs}.]
Proposition~\ref{prop:rec} guarantees that $C_\tau(\delta)$ is uniquely
determined by recursion  (\ref{eqn:Calpha}). By a similar argument as
in the proof of Proposition~\ref{prop:rec} we see that
$C_\tau^{\text{end}}(\delta)$ is given by the formula in Algorithm
\ref{alg:nodecoeffs}. By Lemma~\ref{lem:degree} all contributions of
template collections of type $\tau$ to the node
polynomial $N_\delta(d)$ are in degree $2\delta-\defect(\tau)$ or less. Hence, after shifting $C_\tau^{\text{end}}(\delta)$
by $\defect(\tau)$, their sum is the coefficient vector of $N_{\delta}(d)$.
\end{proof}

To solve recursion  (\ref{eqn:rec}) for a type $\tau$ we make use of the following
(conjectural) structure about $C_\tau(\delta)$ which has been verified for all types $\tau$ with
$\defect(\tau) \le 8$. This refines an observation of L. G\"ottsche
\cite[Remark 4.2 (2)]{Go} about the first $28$ (conjectural) coefficients of the
node polynomial $N_{\delta}(d)$.

\begin{conjecture}
\label{conj:polyn}
All entries of $C_\tau(\delta)$ are of the form
$\frac{3^\delta}{\delta !}$ times a polynomial in $\delta$.
\end{conjecture}

Now, to solve recursion  (\ref{eqn:rec}), we  first extend the natural partial order on the types $\tau$ given by $|\tau| =
\sum_{j \ge 2} \tau_j$ to a linear order with smallest element $\tau =
(0,0,\dots)$. For example, for $N = 4$, we could take
\begin{displaymath}
(0,0,0) < (1,0,0) < (0,1,0) < (0,0,1) < (1,1,0) < (2,0,0) < (3,0,0).
\end{displaymath}
Then solve recursion (\ref{eqn:rec})
for each $\tau$, in increasing order, using the
lowertriangularity of the matrices $M_i$. For example, to
compute the second entry $\frac{3^\delta}{\delta !}p(\delta)$ of
$C_{1,1}(\delta)$ (assuming Conjecture~\ref{conj:polyn}), where
$p(\delta)$ is a polynomial in $\delta$, we need to solve
\begin{displaymath}
\small
C_{1,1}(\delta) =  M_1(2\delta-5) C_{1,1}(\delta-1) + M_2(2\delta-6)
C_{0,1}(\delta-2) + M_3(2\delta-7) C_{1,0}(\delta-3),
\end{displaymath}
or, explicitly,
\begin{displaymath}
\small \begin{bmatrix}
*  \\
\frac{3^\delta}{\delta !}p(\delta) \\
\vdots
 \end{bmatrix}
=
\begin{bmatrix}
* & 0 & 0  \\
* & * & 0 \\
\vdots & \vdots & \ddots  \end{bmatrix}
\begin{bmatrix}
*  \\
\frac{3^{\delta-1}}{(\delta-1)!}p(\delta-1) \\
\vdots  \end{bmatrix}
+
\begin{bmatrix}
* & 0 & 0  \\
* & * & 0 \\
\vdots & \vdots & \ddots   \end{bmatrix}
\begin{bmatrix}
*\\
* \\
\vdots \\
\end{bmatrix}
+
\begin{bmatrix}
* & 0 & 0  \\
* & * & 0 \\
\vdots & \vdots & \ddots  \end{bmatrix}
\begin{bmatrix}
*  \\
*\\
\vdots  \end{bmatrix}.
\end{displaymath}
The $*$-entries in the vectors $C_{0,1}$ and $C_{1,0}$ are known by a previous
computation. The $*$-entries in $M_1$, $M_2$ and $M_3$ are given by (\ref{eqn:Calpha}).
The proof of Lemma \ref{lem:independence} implies that all denominators of
$M_i(a)$ in row $j$ are $a+i-j+2$ or $1$ (after cancellation).
To compute $p(\delta)$, or, more
generally,
the $j$th entry in $C_\tau(\delta)$, we first clear all denominators and then solve the polynomial difference equation with
initial conditions
\begin{equation}
\label{eqn:localrec}
\begin{split}
 (2 \delta  -\defect(\tau)-j+1) 3 p(\delta) &= p(\delta-1) + q(\delta), \\
 p(\delta_0(\tau)) &= C_\tau(\delta_0(\tau)),
\end{split}
\end{equation}
where $q(\delta)$ is a rather complicated polynomial depending on earlier
calculations and $\delta_0(\tau)$ is as in
(\ref{eqn:initialdelta}). One way to solve (\ref{eqn:localrec}) is to bound the degree of
the polynomial $p(\delta)$ and solve the corresponding linear system.

Note that a difference equation of the form (\ref{eqn:localrec}) need
not have a polynomial
solution in general. Conjecture~\ref{conj:polyn} is equivalent to all
recursions (\ref{eqn:localrec}) appearing in Algorithm
\ref{alg:nodecoeffs} to have a polynomial solution.

As in Section~\ref{sec:nodepolys} (Remark \ref{rmk:improvements}),
Algorithm
~\ref{alg:nodecoeffs} can be improved significantly by summing the
 template polynomials $P(\Gamma,k)$ for templates $\Gamma$ with fixed
 $\big( k_{\min}(\Gamma), l(\Gamma), \e(\Gamma)\big)$ in advance.
Algorithm~\ref{alg:nodecoeffs} has been implemented in Maple. Once the templates are known the bottleneck of the algorithm
is the initial value computation which, with an improved
implementation, should be faster than the template
enumeration. Hence we expect Algorithm~\ref{alg:nodecoeffs} to compute the first $14$ terms of $N_\delta(d)$ in reasonable time.

\appendix
\section{Node Polynomials for $\delta \le 14$}
\label{app:newpolys}

An explicit list of $N_\delta(d)$, for $\delta \le 14$, is as
below. These polynomials are given implicitly in
Theorem~\ref{thm:nodepolynomials}. For $\delta \le 8$ this agrees with
\cite[Theorem~3.1]{KP}. For $\delta \le 14$ this coincides with the conjectural
(implicit) formulas of \cite[Remark~2.5]{Go}.

\begin{equation*}
\tiny
\begin{split}
N_0(d) &= 1, \\
N_1(d)
&=3(d-1)^2,\\
N_2(d)
&=\frac{3}{2}(d-1)(d-2)(3d^2-3d-11),\\
N_3(d)
&=\frac{9}{2}d^{6}-27d^{5}+\frac{9}{2}d^{4}+\frac{423}{2}d^{3}-229d^{2}-\frac{829}{2}d+525,\\
N_4(d)
&=\frac{27}{8}d^{8}-27d^{7}+\frac{1809}{4}d^{5}-642d^{4}-2529d^{3}+\frac{37881}{8}d^{2}+\frac{18057}{4}d-8865,\\
N_5(d)
&=\frac{81}{40}d^{10}-\frac{81}{4}d^{9}-\frac{27}{8}d^{8}+\frac{2349}{4}d^{7}-1044d^{6}-\frac{127071}{20}d^{5}+\frac{128859}{8}d^{4}+\frac{59097}{2}d^{3}-\frac{3528381}{40}d^{2}\\
& -\frac{946929}{20}d+153513,\\
N_6(d)
&=\frac{81}{80}d^{12}-\frac{243}{20}d^{11}-\frac{81}{20}d^{10}+\frac{8667}{16}d^{9}-\frac{9297}{8}d^{8}-\frac{47727}{5}d^{7}+\frac{2458629}{80}d^{6}+\frac{3243249}{40}d^{5}\\
& -\frac{6577679}{20}d^{4}-\frac{25387481}{80}d^{3}+\frac{6352577}{4}d^{2}+\frac{8290623}{20}d-2699706,\\
N_7(d)
&=\frac{243}{560}d^{14}-\frac{243}{40}d^{13}-\frac{243}{80}d^{12}+\frac{30861}{80}d^{11}-\frac{38853}{40}d^{10}-\frac{802143}{80}d^{9}+\frac{3140127}{80}d^{8}+\frac{18650493}{140}d^{7}\\
&
-\frac{54903831}{80}d^{6}-\frac{72723369}{80}d^{5}+\frac{124680069}{20}d^{4}+\frac{213537633}{80}d^{3}-\frac{3949576431}{140}d^{2}-\frac{188754021}{140}d\\
& +48016791,\\
N_8(d)
&=\frac{729}{4480}d^{16}-\frac{729}{280}d^{15}-\frac{243}{140}d^{14}+\frac{35721}{160}d^{13}-\frac{25839}{40}d^{12}-\frac{320841}{40}d^{11}+\frac{11847087}{320}d^{10}\\
&
+\frac{170823033}{1120}d^{9}-\frac{6685218}{7}d^{8}-\frac{1758652263}{1120}d^{7}+\frac{1102682031}{80}d^{6}+\frac{59797545}{8}d^{5}-\frac{510928080111}{4480}d^{4}\\
& -\frac{3283674393}{1120}d^{3} +\frac{558215113803}{1120}d^{2}-\frac{3722027733}{56}d-861732459,\\
N_9(d)
&=\frac{243}{4480}d^{18}-\frac{2187}{2240}d^{17}-\frac{729}{896}d^{16}+\frac{121743}{1120}d^{15}-\frac{99549}{280}d^{14}-\frac{824823}{160}d^{13}+\frac{8776593}{320}d^{12}+\frac{74122857}{560}d^{11}\\
&-\frac{2188424421}{2240}d^{10}-\frac{132610923}{70}d^{9}+\frac{11404136871}{560}d^{8}+\frac{2852923401}{224}d^{7}-\frac{3523392270287}{13440}d^{6}\\
&+\frac{4109675615}{448}d^{5}+\frac{261844582229}{128}d^{4}-\frac{2156232149611}{3360}d^{3}-\frac{29528525065861}{3360}d^{2}+\frac{438722045999}{168}d\\
& +15580950065,\\
\end{split}
\end{equation*}
\begin{equation*}
\tiny
\begin{split}
N_{10}(d)
&=\frac{729}{44800}d^{20}-\frac{729}{2240}d^{19}-\frac{729}{2240}d^{18}+\frac{408969}{8960}d^{17}-\frac{746253}{4480}d^{16}-\frac{1932579}{700}d^{15}+\frac{10649961}{640}d^{14}\\
&+\frac{205722099}{2240}d^{13}-\frac{4375229931}{5600}d^{12}-\frac{38815692777}{22400}d^{11}+\frac{30958937073}{1400}d^{10}+\frac{3413568339}{224}d^{9}\\
&-\frac{3624162885799}{8960}d^{8}+\frac{134470136581}{2800}d^{7}+\frac{27023302169081}{5600}d^{6}-\frac{22514488581251}{8960}d^{5}-\frac{811909836973903}{22400}d^{4}\\
&+\frac{253124357071961}{11200}d^{3}+\frac{867510616107447}{5600}d^{2}
-\frac{2800250331071}{40}d-283516631436, \\
N_{11}(d)
&=\frac{2187}{492800}d^{22}-\frac{2187}{22400}d^{21}-\frac{729}{6400}d^{20}+\frac{150903}{8960}d^{19}-\frac{303993}{4480}d^{18}-\frac{56670273}{44800}d^{17}+\frac{47717667}{5600}d^{16}\\
&+\frac{295979589}{5600}d^{15}-\frac{11410430877}{22400}d^{14}-\frac{4051907631}{3200}d^{13}+\frac{52491198663}{2800}d^{12}+\frac{3418059518271}{246400}d^{11}\\
&-\frac{20587006282467}{44800}d^{10}+\frac{2236636275459}{22400}d^{9}+\frac{49175916627959}{6400}d^{8}-\frac{1464110674563}{256}d^{7}\\
&-\frac{1946239824069277}{22400}d^{6}+\frac{3767687640687823}{44800}d^{5}+\frac{14264414890838423}{22400}d^{4}-\frac{940418544772283}{1600}d^{3}\\
&-\frac{168280746183263029}{61600}d^{2}+\frac{5073050867636909}{3080}d
+5187507215325,\\
N_{12}(d)
&=\frac{2187}{1971200}d^{24}-\frac{6561}{246400}d^{23}-\frac{2187}{61600}d^{22}+\frac{496449}{89600}d^{21}-\frac{136809}{5600}d^{20}-\frac{1618623}{3200}d^{19}+\frac{674946837}{179200}d^{18}\\
&+\frac{2321658693}{89600}d^{17}-\frac{893195181}{3200}d^{16}-\frac{34334301951}{44800}d^{15}+\frac{289702847403}{22400}d^{14}+\frac{1245724147341}{123200}d^{13}\\
&-\frac{803786361621603}{1971200}d^{12}+\frac{65497548165237}{492800}d^{11}+\frac{16192295343681}{1792}d^{10}-\frac{792669234543351}{89600}d^{9}\\
&-\frac{9506773589164709}{67200}d^{8}+\frac{6296062244021929}{33600}d^{7}+\frac{11029935159768347}{7168}d^{6}-\frac{582428855393100577}{268800}d^{5}\\
&-\frac{5477484616918678589}{492800}d^{4}+\frac{10067756533588172119}{739200}d^{3}+\frac{4454424013895459501}{92400}d^{2}\\
&-\frac{111952943233924509}{3080}d-95376705265437,\\
N_{13}(d) &=
\frac{6561}{25625600}d^{26}-\frac{6561}{985600}d^{25}-\frac{19683}{1971200}d^{24}+\frac{1620567}{985600}d^{23}-\frac{88209}{11200}d^{22}-\frac{3212703}{17920}d^{21}+\frac{262066023}{179200}d^{20}\\
&+\frac{494726373}{44800}d^{19}-\frac{673360047}{5120}d^{18}-\frac{35350103511}{89600}d^{17}+\frac{20952637821}{2800}d^{16}+\frac{3013479294723}{492800}d^{15}\\
&-\frac{580214902388013}{1971200}d^{14}+\frac{1666286215401123}{12812800}d^{13}+\frac{16384163286402207}{1971200}d^{12}-\frac{909876952033137}{89600}d^{11}\\
&-\frac{7649416285706767}{44800}d^{10}+\frac{25855007471662161}{89600}d^{9}+\frac{65085797443981191}{25600}d^{8}-\frac{108443195356282427}{22400}d^{7}\\
&-\frac{52991400162927629917}{1971200}d^{6}+\frac{1976324604711031517}{39424}d^{5}+\frac{13580753080243105219}{70400}d^{4}\\
&-\frac{73274705967431063281}{246400}d^{3}-\frac{68173290776099374391}{80080}d^{2}+\frac{2813974748454890667}{3640}d +1761130218801033,\\
N_{14}(d) &=
\frac{19683}{358758400}d^{28}-\frac{19683}{12812800}d^{27}-\frac{6561}{2562560}d^{26}+\frac{1751787}{3942400}d^{25}-\frac{4529277}{1971200}d^{24}-\frac{562059}{9856}d^{23}
\\
&
+\frac{398785599}{788480}d^{22}+\frac{5214288411}{1254400}d^{21}-\frac{4860008991}{89600}d^{20}-\frac{63174295089}{358400}d^{19}+\frac{332872084467}{89600}d^{18}\\
&
+\frac{3103879378581}{985600}d^{17}-\frac{4913807521304691}{27596800}d^{16}+\frac{899178800016807}{8968960}d^{15}+\frac{279086438050359453}{44844800}d^{14}\\
&
-\frac{468967272863997483}{51251200}d^{13}-\frac{318443311640108577}{1971200}d^{12}+\frac{328351365725506869}{985600}d^{11}\\
&
+\frac{1120586814080571923}{358400}d^{10}-\frac{9448861028448843949}{1254400}d^{9}-\frac{30880785216736406143}{689920}d^{8}\\
&
+\frac{444525313669622586903}{3942400}d^{7}+\frac{11429038221675466251}{24640}d^{6}-\frac{269709254062572016617}{246400}d^{5}\\
&
-\frac{74660630664748878665353}{22422400}d^{4}+\frac{140531359469510983018159}{22422400}d^{3}+\frac{16863931195154225977601}{1121120}d^{2}\\
& -\frac{64314454486825349085}{4004}d-32644422296329680.
\end{split}
\end{equation*}

\section{Small Severi degrees}
\label{app:smallSevdegs}

Below we list the Severi degrees $N^{d, \delta}$ for $0 \le \delta \le
14$ and $1 \le d \le 13$, which were obtained by Algorithm
  \ref{alg:nodepoly} (also see Remark \ref{rmk:numericalalgo}). Together
  with the node polynomials of Appendix
  \ref{app:newpolys}, this is a full description of all Severi degrees $N^{d, \delta}$ for
$\delta \le 14$, see Theorem \ref{thm:threshold}. The solid line
segments indicate the polynomial threshold $d^*(\delta)$ of $N^{d, \delta}$. The dashed line segments illustrate the
threshold of our Theorem~\ref{thm:threshold}. The Severi degrees $N^{d,
  \delta}$ in italic agree with the Gromov-Witten invariants $N_{d, \frac{(d-1)(d-2)}{2}
- \delta}$, as for $d \ge \delta + 2$, every plane degree $d$ curve with
$\delta$ nodes is irreducible.

\begin{picture}(100,499)
\put(80,-36){

\begin{rotate}{90}
\footnotesize
\begin{tabular}{c|ccccccccc}
& $d = 1$ & 2 & 3 & 4 & 5 & 6 & 7 & 8 & 9 \\
\hline
$N^{d,0}$ &1  & \it 1 & \it 1 & \it 1 & \it 1 & \it 1 & \it 1 & \it
1 & \it 1\\
$N^{d,1}$ & 0 &3 &\it 12 &\it 27 &\it 48 &\it 75 &\it 108 &\it 147 &\it 192\\
\cdashline{2-2}
$N^{d,2}$ & \multicolumn{1}{c:}{0}&0  &21 &\it 225 &\it 882 &\it 2370 &\it 5175 &\it 9891 &\it 17220\\
\cdashline{3-3} 
\cline{2-3} 
$N^{d,3}$ &0 & \multicolumn{1}{c|}{0}&15 &675 & \it 7915 &\it 41310 &\it 145383 &\it 404185 &\it 959115\\
\cdashline{4-4}
$N^{d,4}$ &0 &\multicolumn{1}{c|}{0} &\multicolumn{1}{c:}{0} &666 &36975 &\it 437517
&\it 2667375 &\it 11225145 &\it 37206936 \\
\cline{4-4} 
\cdashline{5-5}
$N^{d,5}$ & 0& 0& \multicolumn{1}{c|}{0} & \multicolumn{1}{c:}{378}&90027
&2931831 &\it 33720354 &\it 224710119 &\it 1068797961\\
\cdashline{6-6}
$N^{d,6}$ &0 & 0& \multicolumn{1}{c|}{0}&105 & \multicolumn{1}{c:}{109781}&12597900 &302280963 & \it 3356773532&\it 23599355991\\
\cline{5-5} 
\cdashline{7-7}
$N^{d,7}$ &0 &0 &0 &\multicolumn{1}{c|}{0} & 65949
&\multicolumn{1}{c:}{34602705} &1950179922 &38232604473 & \it 410453320698\\
\cdashline{8-8}
$N^{d,8}$ & 0 & 0 & 0&\multicolumn{1}{c|}{0} &26136 & 59809860& \multicolumn{1}{c:}{9108238023}& 336507128820&5717863228995\\
\cline{6-6} 
\cdashline{9-9}
$N^{d,9}$ & 0& 0& 0&0 &\multicolumn{1}{c|}{6930} &63338881 &30777542450 &\multicolumn{1}{c:}{2307156326490} &64541125393337\\
\cdashline{10-10}
$N^{d,10}$ &0 &0 &0 &0 & \multicolumn{1}{c|}{945}& 40047888&74808824094
&12372036675723 & \multicolumn{1}{c:}{595034126865816}\\
\cline{7-7} 
$N^{d,11}$ &0 &0 &0 &0 &0 & \multicolumn{1}{c|}{15580020}&129429708147 &51941532804912 &4504735527185481\\
$N^{d,12}$ &0 &0 &0 &0 &0 & \multicolumn{1}{c|}{4361721}&157012934283 &170460529136614 &28096248183844557\\
\cline{8-8} 
$N^{d,13}$ & 0 & 0& 0&0 &0 &918918 & \multicolumn{1}{c|}{131024290671}&
435634878105750 & 144607321488666150 \\
$N^{d,14}$ &0 & 0& 0& 0& 0& 135135 & \multicolumn{1}{c|}{73778495220}&
861893389007280& 614331908473795659 \\
\cline{9-9} 
\end{tabular}
\end{rotate}
}

\put(320,-36){

\begin{rotate}{90}
\footnotesize

\begin{tabular}{c|cccc}
& $d = 10$ & 11 & 12 & 13 \\
\hline
$N^{d,0}$ &\it 1 &\it 1 &\it 1 &\it 1 \\
$N^{d,1}$ &\it 243 & \it 300&\it 363 &\it 432\\
$N^{d,2}$ &\it 27972 &\it 43065&\it 63525&\it 90486\\
$N^{d,3}$ &\it 2029980& \it 3939295& \it 7139823& \it 12245355 \\
$N^{d,4}$ &\it 104285790 &\it 257991042 &\it 579308220 &\it 1203756165\\
$N^{d,5}$ &\it 4037126346 &\it 12886585236&\it 36161763120&\it 91629683271\\
$N^{d,6}$ &\it 122416062018 & \it 510681301550&\it 1807308075111 &\it 5622246678741\\
$N^{d,7}$ &\it 2983927028787 &\it 16491272517465 &\it 74314664917722&\it 285826689019395\\
$N^{d,8}$ &\it 59546865647151&\it 442342707233400& \it 2563893146687301&\it 12282025653769635\\
$N^{d,9}$ & 985875034961260&\it 9996104553443766 &\it 75318503191523715 &\it 452837863689428040\\
$N^{d,10}$ &13675748317151382 &192382226805469707 &\it 1905520429287295623 &\it 14494356139317787773\\
\cdashline{2-2}
$N^{d,11}$ & \multicolumn{1}{c:}{160116000544437849}& 3179784684983704875&41891284347920817345 &\it 406515611290526234685\\
\cdashline{3-3}
$N^{d,12}$ &1590895229889323034 &\multicolumn{1}{c:}{45433391588342055421} & 806014803108359459265&10065752539264069119357\\
\cdashline{4-4}
$N^{d,13}$ & 13467927464624076876 & 564072316378226731551
&\multicolumn{1}{c:}{13651752114752769885861 } & 221404946495996659008375\\
\cdashline{5-5}
$N^{d,14}$ &97415160821449934985 &6109881487479049410675 &
204507963635254531972650 
&\multicolumn{1}{c:}{4348333391475310314325875 }\\
\end{tabular}
\end{rotate}
} 
\end{picture}

\bibliographystyle{amsplain}
\bibliography{References_Florian}
\label{sec:biblio}

  \end{document}